\documentclass[11pt,a4paper]{article}
\usepackage[letterpaper,top=2.5cm,bottom=2.5cm,left=2.5cm,right=2.5cm,marginparwidth=1.75cm]{geometry}
\usepackage{color}
\usepackage{amsmath,amsfonts,amsthm,amssymb}
\usepackage{mathrsfs}
\usepackage{hyperref}
\usepackage{float}
\usepackage{enumitem}
\usepackage{graphicx,fancyhdr}
\usepackage{mathrsfs,float}
\usepackage{multirow}
\usepackage{dsfont}
\usepackage{subfigure}
\usepackage{booktabs}
\usepackage{yfonts}
\usepackage{caption}
\usepackage{bbm}
\makeatletter

\numberwithin{equation}{section}
\newtheorem{theorem}{Theorem}[section]
\newtheorem{lemma}[theorem]{Lemma}

\theoremstyle{definition}
\newtheorem{definition}[theorem]{Definition}
\newtheorem{remark}[theorem]{Remark}
\newtheorem{example}{Example}

\newcommand{\dom}{\operatorname{dom}}

\newcommand{\diff}{\,\mathrm{d}}

\begin{document}

\title{ Logarithmic stability for recovering initial data of fractional heat equations from thin-set observations}

\author{
Kai Yu \thanks{School of Mathematics and Statistics, Ningbo University, 818 Fenghua Road, Ningbo 315211, China.} \and
Zhiyuan Li \thanks{Corresponding author, School of Mathematics and Statistics, Ningbo University, 818 Fenghua Road, Ningbo 315211, China. E-mail:lizhiyuan@nbu.edu.cn} 
}
\date{}

\maketitle

\begin{abstract}

We study the inverse problem of recovering initial data for fractional heat equations on bounded domains from observations taken on two types of possibly Lebesgue‑null sets: the first consists of general thin sets whose Hausdorff dimension exceeds \(n-1\); the second consists of specially constructed sets of zero Hausdorff dimension, built from algebraic irrational points and rapidly accumulating sequences.
By extending thin-set observability inequalities from the classical heat equation to the fractional setting and employing a unified spectral inequality with exponent \(\beta\in(0,s)\), we establish a quantitative observability estimate for every fractional exponent \(s>1/2\). For general thin sets one has \(\beta=1/2\); for the special zero‑dimensional sets any \(\beta\in(1/2,s)\) is admissible, yet the short‑time observability cost always remains of exponential type, with the rate governed by \(\beta\) relative to \(s\).
Under an a priori smoothness assumption, we also prove a logarithmic stability estimate.
We further design a regularised least‑squares reconstruction algorithm and provide a conditional convergence analysis, showing that the reconstruction error is bounded by the sum of a spectral truncation term and a logarithmic term dictated by the continuous stability, and that the error decays logarithmically as the noise level tends to zero. 
Numerical simulations using very few observation points confirm the feasibility of the proposed approach.



\end{abstract}


\section{Introduction}

Let \(\Omega\subset\mathbb R^n\) be a bounded domain with  smooth boundary \(\partial\Omega\), and let \(T>0\). For the classical heat equation
\[
\partial_t u-\Delta u=0,\qquad (x,t)\in\Omega\times(0,T),
\]
with homogeneous Dirichlet boundary conditions, it is by now classical that observability and null controllability hold from any nonempty interior open set \(\mathcal O\subset\Omega\). The pioneering works of Russell~\cite{Russell78}, Lebeau and Robbiano~\cite{LebeauRobbiano95}, Lebeau and Zuazua~\cite{LebeauZuazua98} established that the observability cost is of exponential type \(e^{C/T}\) for small times. This theory has been subsequently extended to boundary observations~\cite{Biccari18}, observations on measurable sets~\cite{PhungWang13} and networks of one-dimensional heat equations~\cite{DagerZuazua06}; we refer to the monographs~\cite{Coron07,TucsnakWeiss09} and the survey~\cite{Zuazua07} for comprehensive overviews.

The closely related inverse problem of recovering the initial state of a parabolic equation from partial observations is severely ill-posed, since the heat semigroup exponentially damps high-frequency information. Consequently, one cannot in general expect stability stronger than logarithmic type, and conditional stability estimates usually require an a priori regularity bound on the unknown initial datum. Fundamental results on uniqueness and logarithmic stability for backward parabolic problems can be found in the monograph of Isakov~\cite{Isakov06} and in the classical works of John~\cite{John60} and Miranker~\cite{Miranker61}. More specifically, Klibanov~\cite{Klibanov06} established estimates for unknown initial conditions of parabolic equations from lateral Cauchy data, while Klibanov and Tikhonravov~\cite{KlibanovTikhonravov07} extended this approach to parabolic equations and inequalities in infinite domains. Bourgeois~\cite{Bourgeois17} proved a logarithmic stability estimate for the heat equation with lateral Cauchy data and applied it to the identification of the initial condition. Conditional stability and numerical reconstruction for the recovery of the initial temperature were studied by Li, Yamamoto and Zou~\cite{LiYamamotoZou08}. For parabolic systems, Garc{\'i}a and Takahashi~\cite{GarciaTakahashi11} investigated inverse initial-state problems through their connection with null-controllability.

In recent decades, fractional diffusion equations have attracted considerable attention because they provide effective models for anomalous diffusion, long-range interactions, and non-local effects in physics, biology, and finance~\cite{BucurValdinoci16,Mainardi22,MetzlerKlafter00}. In this paper, we focus on the fractional heat equation
\[
\partial_t u + (-\Delta)^s u = 0,\qquad s\in(0,1],
\]
where \((-\Delta)^s\) is the fractional power of the Dirichlet Laplacian, defined spectrally by \((-\Delta)^s\varphi_k = \lambda_k^s\varphi_k\). Thus the eigenfunctions are exactly those of the classical Laplacian, while the eigenvalues are raised to the power \(s\). Well-posedness and regularity for such equations have been studied in~\cite{CaffarelliSilvestre07,CapellaDavilaDupaigne11,RosOtonSerra14}. Control and observability properties differ substantially from the classical case: in one space dimension, Micu and Zuazua~\cite{MicuZuazua06} proved that null controllability from open sets holds precisely when \(s>1/2\); the higher-dimensional result was established by Koenig~\cite{Koenig20}. Other contributions on controllability and observability of fractional diffusion can be found in~\cite{BiccariWarmaZuazua18,LuZuazua16}. Closely related to these controllability properties are inverse problems for fractional diffusion equations, which have also been the subject of intensive research. Inverse problems for time-fractional equations are discussed in~\cite{JinRundell15,LiLiuYamamoto17,SakamotoYamamoto11}. For space-fractional diffusion equations, reconstruction and stability results have been obtained in~\cite{LinRailo25,ZhaoLiu14,ZhengZhang18}. Recovery of parameters in time-space fractional diffusion equations is studied, for instance, in~\cite{AliAziz18,LiWei18,TatarUlusoy15,ZhangJia18}.

A remarkable recent breakthrough by Green, Le Balc'h, Martin, and Orsoni~\cite{Green25} has significantly relaxed the geometric requirements on the observation set for the classical heat equation. Instead of an open set, it is sufficient to observe the solution on a thin set \(\omega\subset\Omega\) whose Hausdorff dimension is strictly larger than \(n-1\); such a set may have zero \(n\)-dimensional Lebesgue measure. 
The key ingredient in~\cite{Green25} is a spectral inequality for low-frequency projections of the Dirichlet Laplacian. Denoting by
\[
\Pi_\Lambda u = \sum_{\lambda_k\le\Lambda} \langle u,\varphi_k\rangle_{L^2(\Omega)}\,\varphi_k
\]
the orthogonal projection onto the span of all eigenfunctions with eigenvalues not exceeding \(\Lambda\), the result states that under a positive Hausdorff content condition \(\mathcal C_H^{\,n-1+\nu}(\omega)>0\) for \(\nu\in (0,1)\), one has
\[
\|\Pi_\Lambda u\|_{L^2(\Omega)} \le C e^{C\sqrt{\Lambda}} \sup_{x\in\omega}|\Pi_\Lambda u(x)| .
\]
Combined with the Lebeau–Robbiano method, this yields an observability estimate
\[
\|u(\cdot,T)\|_{L^\infty(\Omega)} \le C e^{C/T} \int_0^T \sup_{x\in\omega}|u(x,t)|\,\diff t .
\]
This result is essentially sharp: nodal sets of eigenfunctions have Hausdorff dimension \(n-1\) and fail to observe the heat flow. In addition,~\cite{Green25} constructed special zero-dimensional observation sets, built from algebraic irrational points and rapidly accumulating sequences, for which an improved spectral inequality holds with exponent \(\beta\in(1/2,1)\) instead of \(1/2\).

Thin-set observability is particularly relevant for applications where measurements are taken on curves, fibres, or sparse sensor arrays, rather than over full open regions. The observation norm is naturally taken as \(\int_0^T\sup_{x\in\omega}|u(x,t)|\,\diff t\); although non-Hilbertian, it is well adapted to zero-measure sets and, when \(\omega\) is open, is essentially equivalent to the standard \(L^2\) norm. For the associated backward heat problem, the stability type remains logarithmic~\cite{Green25,HuangWang26}, but the constants additionally involve the Hausdorff content of \(\omega\).

The purpose of this article is to extend the thin-set observability theory to the fractional heat equation. Because the fractional operator \((-\Delta)^s\) shares the eigenfunctions of \(-\Delta\), the spectral inequalities for \(\Pi_\Lambda\) remain unchanged. The crucial difference lies in the high-frequency dissipation, which becomes \(e^{-t\Lambda^s}\) instead of \(e^{-t\Lambda}\). To unify the treatment, suppose an observation set \(\omega\) satisfies a spectral inequality of the form
\[
\|\Pi_\Lambda f\|_{L^2(\Omega)} \le C e^{C\Lambda^\beta} \sup_{x\in\omega}|\Pi_\Lambda f(x)|,\qquad \beta\in(0,1).
\]
For general thin sets one has \(\beta=1/2\), which imposes \(s>1/2\) and yields the cost \(\exp\bigl(C/T^{1/(2s-1)}\bigr)\). For the special zero-dimensional sets, any \(\beta\in(1/2,1)\) is available; choosing \(\beta\in(1/2,s)\) again requires \(s>1/2\) and produces the corresponding cost.

Let \(Q=\Omega\times(0,T)\), \(\Sigma=\partial\Omega\times(0,T)\) and \(s\in(1/2,1]\), we study both the homogeneous fractional heat equation
\begin{equation}\label{eq:frac-heat}
\begin{cases}
\partial_t u+(-\Delta)^s u=0, & \text{in } Q,\\
u(0)=u_0, & \text{in } \Omega,\\
u=0, & \text{on } \Sigma,
\end{cases}
\end{equation}
and its nonhomogeneous analogue 
\begin{equation}\label{eq:frac-heat-source}
\begin{cases}
\partial_t u+(-\Delta)^s u=f(x,t), & \text{in } Q,\\
u(0)=u_0, & \text{in } \Omega,\\
u=0, & \text{on } \Sigma,
\end{cases}
\end{equation}
with a known source term \(f\). In the latter case, writing the solution as \(u=u^0+u^f\), where \(u^0\) solves the homogeneous problem \eqref{eq:frac-heat} with initial data \(u_0\) and \(u^f\) is the known Duhamel contribution.

Our main results are:
\begin{enumerate}

\item[(i)] 
We prove an observability inequality for the fractional heat equation \eqref{eq:frac-heat} that applies to any observation set satisfying a spectral inequality with exponent \(\beta<s\) (Theorem~\ref{thm:obs}). This framework encompasses both general thin sets with Hausdorff dimension larger than \(n-1\) (for which \(\beta=1/2\), requiring \(s>1/2\)) and the specially constructed zero-dimensional sets built from algebraic irrational points and rapidly accumulating sequences (for which any \(\beta\in(1/2,s)\) can be chosen).

\item[(ii)] 
We apply the unified observability inequality to the nonhomogeneous equation \eqref{eq:frac-heat-source} with a known source term \(f\). By subtracting the computable Duhamel contribution of \(f\) from the measurements, the recovery of the initial data is reduced to the homogeneous problem. Under an a priori bound \(\|(-\Delta)^r u_0\|_{L^2(\Omega)}\le M\), this yields a logarithmic stability estimate (Theorem~\ref{thm:stability}) of the form \(C M \bigl[\log(e + M/\varepsilon)\bigr]^{-r/s}\). The exponent \(r/s\) reflects the balance between the backward amplification \(e^{T\Lambda^s}\) of low frequencies and the high-frequency tail controlled by \(M\Lambda^{-r}\), while the geometric complexity of the observation set (encoded in the spectral exponent \(\beta\)) affects only the multiplicative constant.

\item[(iii)] 
We develop a regularised least‑squares scheme based on truncated spectral expansions with a penalty on the higher eigenmodes. A convergence analysis shows that the reconstruction error is controlled by the truncation level, the noise intensity, and the continuous stability estimate. Numerical experiments in one and two space dimensions demonstrate that initial data can be accurately recovered from very few specially chosen observation points, even in the presence of moderate noise.

\end{enumerate}

The paper is organised as follows.
Section~\ref{sec2} introduces the functional setting, presents the spectral inequalities for the two types of observation sets, and states both the unified observability inequality and the logarithmic stability estimate.
Section~\ref{sec3} contains the proof of the observability inequality.
Section~\ref{sec4} provides the proof of the logarithmic stability theorem.
Section~\ref{sec5} is devoted to the numerical experiments: a unified regularised least-squares algorithm is described, a conditional convergence analysis is given, and reconstruction results in one and two space dimensions are presented.
Section~\ref{sec6} summarises the main contributions of the paper.


\section{Preliminaries}\label{sec2}

Let \(\Omega\subset\mathbb R^n\) be a bounded domain with \(C^2\) boundary. We denote by \(A=-\Delta\) the Dirichlet Laplacian on \(L^2(\Omega)\), with domain \(\dom(A)=H^2(\Omega)\cap H_0^1(\Omega)\). The operator \(A\) is positive, self-adjoint, and has compact resolvent. Consequently, there exists an orthonormal basis \(\{\varphi_k\}_{k\ge1}\subset L^2(\Omega)\) of eigenfunctions of \(A\), associated with a nondecreasing sequence of eigenvalues
\[
0<\lambda_1\le\lambda_2\le\cdots,\qquad \lambda_k\to+\infty,
\]
that is,
\[
A\varphi_k=\lambda_k\varphi_k,\qquad \langle\varphi_k,\varphi_\ell\rangle_{L^2(\Omega)}=\delta_{k\ell}.
\]

For \(s>0\), the fractional power \(A^s\) is defined via the spectral calculus:
\[
A^s u = \sum_{k=1}^\infty \lambda_k^s\langle u,\varphi_k\rangle_{L^2(\Omega)}\varphi_k,
\]
with domain
\[
\dom(A^s)=\Bigl\{u\in L^2(\Omega):\sum_{k=1}^\infty \lambda_k^{2s}|\langle u,\varphi_k\rangle_{L^2(\Omega)}|^2<\infty\Bigr\}.
\]
The semigroup generated by \(-A^s\) is analytic and acts on \(L^2(\Omega)\) as
\[
e^{-tA^s}u = \sum_{k=1}^\infty e^{-t\lambda_k^s}\langle u,\varphi_k\rangle_{L^2(\Omega)}\varphi_k,\qquad t>0.
\]

For \(\Lambda>0\), the low-frequency spectral projection is
\[
\Pi_\Lambda u = \sum_{\lambda_k\le\Lambda} \langle u,\varphi_k\rangle_{L^2(\Omega)}\varphi_k .
\]
Since each \(\Pi_\Lambda u\) is a finite linear combination of smooth eigenfunctions, the quantity \(\sup_{x\in\omega}|\Pi_\Lambda u(x)|\) is well defined for any set \(\omega\subset\Omega\). This is crucial because our observation sets may have zero Lebesgue measure, so that \(L^2\)-based norms over \(\omega\) are not appropriate.

For a set \(E\subset\mathbb R^n\) and \(d\ge0\), the \(d\)-dimensional Hausdorff content is
\[
\mathcal C_{\mathcal H}^d(E)=\inf\Bigl\{\sum_j r_j^d : E\subset\bigcup_j B(x_j,r_j)\Bigr\},
\]
and the Hausdorff dimension is \(\dim_{\mathcal H}(E)=\inf\{d\ge0:\mathcal C_{\mathcal H}^d(E)=0\}\).

\subsection{Admissible observation sets and spectral inequalities}

The geometric backbone of our analysis is a spectral inequality for low-frequency functions. We state it in a unified form.

\begin{definition}[Admissible observation set]\label{def:admissible-set}
Let \(\omega\subset\Omega\) and \(\beta\in(0,1]\). We say that \(\omega\) is an \emph{admissible observation set with exponent \(\beta\)} if there exists a constant \(C>0\) such that for all \(\Lambda>0\) and all \(u\in L^2(\Omega)\),
\begin{equation}\label{eq:unified-spectral}
\|\Pi_\Lambda u\|_{L^2(\Omega)}\le C e^{C\Lambda^\beta}\sup_{x\in\omega}|\Pi_\Lambda u(x)|.
\end{equation}
\end{definition}

The exponent \(\beta\) controls the cost of recovering a low-frequency function from its values on \(\omega\). In the Lebeau–Robbiano method this cost will be balanced against the high-frequency dissipation of the fractional heat semigroup, which is of order \(e^{-t\Lambda^s}\). Thus the natural threshold is \(\beta<s\).

We now describe two classes of admissible observation sets.

\subsubsection{General thin sets}

The first class consists of sets whose Hausdorff dimension is strictly larger than \(n-1\). Such sets were introduced in~\cite{Green25}.

Let \(\omega\subset\Omega\), we assume that there exists \(\nu>0\) such that
\begin{equation}\label{eq:content-cond}
\mathcal C_{\mathcal H}^{\,n-1+\nu}(\omega)>0.
\end{equation}
This condition implies \(\dim_{\mathcal H}(\omega)\ge n-1+\nu>n-1\). It is weaker than requiring \(\omega\) to contain an open subset of \(\Omega\). Indeed, \(\omega\) may have zero \(n\)-dimensional Lebesgue measure whenever its Hausdorff dimension is strictly smaller than \(n\).

\begin{lemma}[see \cite{Green25}]\label{lem:spectral-general}
If \(\omega\subset\Omega\) satisfies \eqref{eq:content-cond}, then \(\omega\) is an admissible observation set with exponent \(\beta=1/2\). More precisely, there exists a constant \(C>0\) depending on \(\Omega\), \(\nu\), and the Hausdorff content of \(\omega\) such that
\begin{equation}\label{eq:spectral-general}
\|\Pi_\Lambda u\|_{L^2(\Omega)}\le C e^{C\sqrt{\Lambda}}\sup_{x\in\omega}|\Pi_\Lambda u(x)|,\qquad \Lambda>0,\; u\in L^2(\Omega).
\end{equation}
\end{lemma}

\begin{remark}
The condition \(\dim_{\mathcal H}(\omega)>n-1\) is sharp for a general theorem of this type. If the observation set is contained in a nodal set of a nontrivial eigenfunction, then it cannot determine all low-frequency components. Since nodal sets typically have Hausdorff dimension \(n-1\), one cannot expect an estimate such as \eqref{eq:spectral-general} for arbitrary sets of Hausdorff dimension \(n-1\).
\end{remark}

\subsubsection{Special zero-dimensional observation sets}

The Hausdorff content condition \eqref{eq:content-cond} excludes sets of Hausdorff dimension zero. Nevertheless, certain highly structured zero-dimensional sets still admit spectral inequalities, with an exponent \(\beta>1/2\) instead of \(1/2\). These sets exploit both arithmetic properties of specific points (e.g., algebraic irrationality) and the rapid accumulation of well-chosen sequences.

We describe the model construction on the cube \(\Omega=(0,1)^n\), \(n\in \mathbb{N}\). For \(\alpha>0\), let
\[
\omega_\alpha=\{i^{-\alpha}:i=1,2,\dots\}\subset(0,1),
\]
and \(x_0\in(0,1)\) be an algebraic irrational point of degree larger than one. The observation set \(\omega\) is taken as
\begin{equation}\label{eq:obs_set}
\omega =\omega_{\alpha}^n=\underbrace{\omega_\alpha\times\cdots\times\omega_\alpha}_{n\text{ times}}\subset (0,1)^n \quad \;\text{or}\; \quad \omega =\{x_0\}\times\omega_\alpha^{\,n-1}  \subset (0,1)^n .
\end{equation}
Both choices are countable, hence \(\dim_{\mathcal H}(\omega)=0\). In particular, when \(n=1\), \(\omega =\omega_{\alpha} \mbox{ or } \{x_0\}\). On \((0,1)^n\), the Dirichlet eigenfunctions are
\[
\varphi_{\mathbf k}(x)=(\sqrt2)^n\prod_{j=1}^n\sin(k_j\pi x_j),\qquad \mathbf k=(k_1,\dots,k_n)\in\mathbb N^n,
\]
with eigenvalues \(\lambda_{\mathbf k}=\pi^2\sum_{j=1}^n k_j^2\). Hence, if \(\lambda_{\mathbf k}\le\Lambda\) then \(k_j\le\sqrt{\Lambda}/\pi\) for each \(j\). Consequently, every function in the range of \(\Pi_\Lambda\) is a trigonometric polynomial whose degree in each variable is bounded by \(C\sqrt{\Lambda}\).

\begin{lemma}[see \cite{Green25}]\label{lem:spectral-special}
Let \(\Omega=(0,1)^n\) for \(n\ge1\) and \(\omega\) be either the Cartesian product \(\omega_\alpha^n\) or the mixed set \(\{x_0\}\times\omega_\alpha^{\,n-1}\) described in \eqref{eq:obs_set}, where \(x_0\in(0,1)\) is an algebraic irrational point of degree larger than one and \(\omega_\alpha=\{i^{-\alpha}\}_{i\ge1}\) with \(\alpha>0\). Then for every \(\beta\in(1/2,1)\) there exists a constant \(C>0\), depending only on \(\Omega\), \(x_0\), \(\alpha\), and \(\nu\), such that
\begin{equation}\label{eq:unified-spectral-ineq}
\|\Pi_\Lambda u\|_{L^2(\Omega)}\le C e^{C\Lambda^\beta}\sup_{x\in\omega}|\Pi_\Lambda u(x)|,\qquad \Lambda>0,\; u\in L^2(\Omega).
\end{equation}
\end{lemma}

The proof of Lemma~\ref{lem:spectral-special} differs from that of Lemma~\ref{lem:spectral-general}. It relies on a one-dimensional Remez-type inequality for trigonometric polynomials evaluated on the discrete set \(\omega_\alpha\), together with a Cartesian product argument in \((0,1)^n\). The essential feature of this inequality is that it holds for every exponent \(\beta\in(1/2,1)\), but not for the endpoint \(\beta=1/2\). As a result, for the fractional heat equation we may fix any \(\beta\in(1/2,s)\); this is possible precisely when \(s>1/2\).

\subsection{Observability inequality}

We now convert the spectral inequality into an observability estimate for the fractional heat equation
\begin{equation}\label{eq:frac-heat-prelim}
\partial_t u + A^s u =0,\quad u(0)=u_0.
\end{equation}
Assume that \(\omega\) is an admissible observation set with exponent \(\beta\in(0,s)\) in the sense of Definition~\ref{def:admissible-set}, i.e.,
\[
\|\Pi_\Lambda u\|_{L^2(\Omega)}\le C e^{C\Lambda^\beta}\sup_{x\in\omega}|\Pi_\Lambda u(x)|.
\]
The high-frequency dissipation is given by the semigroup estimate
\[
\|(I-\Pi_\Lambda)e^{-tA^s}u\|_{L^2(\Omega)}\le e^{-t\Lambda^s}\|u\|_{L^2(\Omega)},
\]
which follows directly from the spectral expansion. The Lebeau–Robbiano method balances the spectral cost \(e^{C\Lambda^\beta}\) with the dissipative factor \(e^{-t\Lambda^s}\); the necessary condition is \(\beta<s\). Choosing \(\Lambda\sim t^{-1/(s-\beta)}\) leads to an observability cost of order \(\exp\bigl(C/T^{\beta/(s-\beta)}\bigr)\).

\begin{theorem}[Unified thin-set observability]\label{thm:obs}
Let \(s\in(1/2,1]\) and let \(\omega\) be an admissible observation set \eqref{eq:unified-spectral} with exponent \(\beta\in(0,s)\). Then there exists a constant \(C>0\) such that for every \(T>0\) and every \(u_0\in L^2(\Omega)\),
\begin{equation}\label{eq:obs-beta-for-stab}
\|u(\cdot,T)\|_{L^2(\Omega)}\le C\exp\Bigl(\frac{C}{T^{\beta/(s-\beta)}}\Bigr)\int_0^T\sup_{x\in\omega}|u(x,t)|\,\diff t,    
\end{equation}
where \(u(x,t)=e^{-tA^s}u_0\) solves \eqref{eq:frac-heat-prelim}. The constant \(C\) depends on \(\Omega\), \(s\), \(\beta\), and the constants in \eqref{eq:spectral-general} or \eqref{eq:unified-spectral-ineq}.
\end{theorem}

\begin{remark}
For general thin sets satisfying \eqref{eq:content-cond} we have \(\beta=1/2\); hence the theorem applies for \(s>1/2\) and gives the cost \(\exp\bigl(C/T^{1/(2s-1)}\bigr)\). For the special zero-dimensional sets of Lemma~\ref{lem:spectral-special}, one may fix any \(\beta\in(1/2,s)\) and obtain a similar estimate with the corresponding exponent. Moreover, using the \(L^2\)-\(L^\infty\) smoothing of the semigroup, one also obtains an \(L^\infty\) terminal estimate at the expense of an additional multiplicative constant, but the \(L^2\) version is sufficient for the inverse problem studied below.
\end{remark}

\subsection{Logarithmic stability for the inverse problem}

We now apply the observability inequality to the recovery of the initial data. Let \(\omega\) be an admissible observation set with exponent \(\beta\in(0,s)\), so that the estimate of Theorem~\ref{thm:obs} holds.

\begin{theorem}[Logarithmic stability]\label{thm:stability}
Let \(s\in(1/2,1]\), \(r>0\), \(M>0\), and let \(\omega\) be an admissible observation set \eqref{eq:unified-spectral} with exponent \(\beta\in(0,s)\). Assume that \(u_0,\widetilde u_0\in\dom(A^r)\) satisfy
\[
\|A^r u_0\|_{L^2(\Omega)}+\|A^r\widetilde u_0\|_{L^2(\Omega)}\le M .
\]
Let \(u, \widetilde u\) be the solution of \eqref{eq:frac-heat-source} with the initial data \(u_0,\tilde u_0\), respectively, and define the observation error
\[
\varepsilon = \int_0^T\sup_{x\in\omega}|u(x,t)-\widetilde u(x,t)|\,\diff t .
\]
Then there exists a constant \(C>0\), depending on \(\Omega\), \(\omega\), \(s\), \(\beta\), \(r\), \(T\), and the constants in the spectral inequality, such that
\begin{equation*}\label{eq:log-stab}
\|u_0-\widetilde u_0\|_{L^2(\Omega)}\le C M\Bigl[\log\Bigl(e+\frac{M}{\varepsilon}\Bigr)\Bigr]^{-r/s}.
\end{equation*}
When \(\varepsilon=0\), the right-hand side is interpreted as \(0\), so uniqueness holds.
\end{theorem}


\section{Observability for the thin-set}\label{sec3}
\begin{proof}[\bf Proof of Theorem~\ref{thm:obs}]
The proof follows the Lebeau--Robbiano strategy, combining the unified spectral inequality \eqref{eq:unified-spectral-ineq} with the high-frequency dissipation of the fractional heat semigroup.

For the semigroup \(e^{-tA^s}\), we use the ultracontractive estimate
\begin{equation}\label{eq:ultra-beta}
\|e^{-tA^s}u\|_{L^\infty(\Omega)} \le C_0 t^{-n/(4s)}\|u\|_{L^2(\Omega)},\qquad t>0.
\end{equation}
This follows from subordination to the classical heat semigroup and the standard \(L^2\)--\(L^\infty\) bound for \(e^{-tA}\).

Fix \(0\le t_1<t_2\le T\), let \(\tau=t_2-t_1\), and set \(v(x,t)=u(x,t_1+t)\), \(0\le t\le\tau\).
For any \(\Lambda>0\) decompose
\[
v(x,t)=\Pi_\Lambda v(x,t)+(I-\Pi_\Lambda)v(x,t).
\]
Applying the spectral inequality \eqref{eq:unified-spectral-ineq} to \(\Pi_\Lambda v(t)\) gives
\begin{equation}\label{eq:low-local-beta}
\|\Pi_\Lambda v(\cdot,t)\|_{L^2(\Omega)}\le C_1 e^{C_1\Lambda^\beta}\sup_{x\in\omega}|\Pi_\Lambda v(x,t)|,\qquad 0<t\le\tau.
\end{equation}
Using the spectral expansion of the semigroup,
\[
(I-\Pi_\Lambda)v(x,t)=e^{-tA^s}(I-\Pi_\Lambda)u(x,t_1)
=\sum_{\lambda_k>\Lambda} e^{-t\lambda_k^s}\langle u(x,t_1),\varphi_k\rangle\varphi_k .
\]
By orthonormality,
\[
\begin{aligned}
\|(I-\Pi_\Lambda)v(\cdot,t)\|_{L^2(\Omega)}^2=\sum_{\lambda_k>\Lambda} e^{-2t\lambda_k^s}|\langle u(x,t_1),\varphi_k\rangle|^2 \le e^{-2t\Lambda^s}\|u(\cdot,t_1)\|_{L^2(\Omega)}^2 .
\end{aligned}
\]
Taking square roots, then the high-frequency part satisfies 
\begin{equation}\label{eq:high-L2-beta}
\|(I-\Pi_\Lambda)v(\cdot,t)\|_{L^2(\Omega)}\le e^{-t\Lambda^s}\|u(\cdot,t_1)\|_{L^2(\Omega)}.
\end{equation}

Now we also need to estimate the supremum of the high-frequency part on the observation set.  
Using the semigroup property, we write
\[
(I-\Pi_\Lambda)v(x,t) = e^{-(t/2)A^s}\bigl[(I-\Pi_\Lambda)v(x,t/2)\bigr],\qquad 0<t\le\tau.
\]
Applying the ultracontractive estimate \eqref{eq:ultra-beta} and the high-frequency \(L^2\) bound \eqref{eq:high-L2-beta} (with \(t/2\) in place of \(t\)), we obtain for any \(t\in[\tau/2,\tau]\)
\[
\begin{aligned}
\|(I-\Pi_\Lambda)v(\cdot,t)\|_{L^\infty(\Omega)}
&\le C_0 (t/2)^{-n/(4s)} \|(I-\Pi_\Lambda)v(\cdot,t/2)\|_{L^2(\Omega)} \\
&\le C_0 (t/2)^{-n/(4s)} e^{-(t/2)\Lambda^s} \|u(\cdot,t_1)\|_{L^2(\Omega)} .
\end{aligned}
\]
Since \(t\ge \tau/2\), we have \((t/2)^{-n/(4s)}\le (\tau/4)^{-n/(4s)}\) and \(e^{-(t/2)\Lambda^s}\le e^{-(\tau/4)\Lambda^s}\).  
Therefore there exist constants \(C_2>0\) and \(c_0\in(0,1)\), depending only on \(n\) and \(s\), such that
\[
\sup_{x\in\omega}|(I-\Pi_\Lambda)v(x,t)|
\le \|(I-\Pi_\Lambda)v(\cdot,t)\|_{L^\infty(\Omega)}
\le C_2 \tau^{-n/(4s)} e^{-c_0\tau\Lambda^s} \|u(\cdot,t_1)\|_{L^2(\Omega)},\quad t\in[\tau/2,\tau].
\]

Now bound the low-frequency observation pointwise:
\[
\begin{aligned}
\sup_{x\in\omega}|\Pi_\Lambda v(x,t)|
&\le
\sup_{x\in\omega}|v(x,t)|
+
\sup_{x\in\omega}|(I-\Pi_\Lambda)v(x,t)|
\\
&\le
\sup_{x\in\omega}|u(x,t_1+t)|
+
C_2 \tau^{-n/(4s)}
e^{-c_0\tau\Lambda^s}
\|u(\cdot,t_1)\|_{L^2(\Omega)}
\end{aligned}
\]
for \(t\in[\tau/2,\tau]\). Insert this into \eqref{eq:low-local-beta} and combine with \eqref{eq:high-L2-beta} to get, for \(t\in[\tau/2,\tau]\),
\begin{equation}\label{eq:pointwise}
\begin{aligned}
\|u(\cdot,t_1+t)\|_{L^2(\Omega)}
&\le C_1 e^{C_1\Lambda^\beta} \sup_{x\in\omega}|u(x,t_1+t)| \\
&\quad + \Bigl(C_1C_2 \tau^{-n/(4s)} e^{C_1\Lambda^\beta-c_0\tau\Lambda^s}
+ e^{-t\Lambda^s}\Bigr) \|u(\cdot,t_1)\|_{L^2(\Omega)}.
\end{aligned}    
\end{equation}
Since \(t\ge\tau/2\), we have \(e^{-t\Lambda^s}\le e^{-\frac{\tau}{2}\Lambda^s}\).
Because the semigroup is contractive, \(\|u(\cdot,t_2)\|_{L^2(\Omega)}\le \|u(\cdot,t_1+t)\|_{L^2(\Omega)}\) for \(t\in[\tau/2,\tau]\).
Integrating the inequality \eqref{eq:pointwise} over \(t\in[\tau/2,\tau]\) and dividing by the length \(\tau/2\) yields
\begin{equation}\label{eq:single-step-beta}
\begin{aligned}
\|u(\cdot,t_2)\|_{L^2(\Omega)}
&\le \frac{2C_1}{\tau} e^{C_1\Lambda^\beta} \int_{t_1}^{t_2} \sup_{x\in\omega}|u(x,t)|\,\diff t \\
&\quad + \Bigl(C_1C_2 \tau^{-n/(4s)} e^{C_1\Lambda^\beta-c_0\tau\Lambda^s}
+ e^{-\frac{\tau}{2}\Lambda^s}\Bigr) \|u(\cdot,t_1)\|_{L^2(\Omega)}.
\end{aligned}
\end{equation}

The condition \(s>\beta\) allows us to balance the exponentials. Set
\[
p = \frac{\beta}{s-\beta}>0,\qquad \Lambda = M\tau^{-\frac{1}{s-\beta}},
\]
with a large constant \(M>0\) to be chosen. Then
\[
\Lambda^\beta = M^\beta \tau^{-p},\qquad \tau\Lambda^s = M^s \tau^{-p}.
\]
Consequently,
\[
c_0\tau\Lambda^s - C_1\Lambda^\beta = (c_0 M^s - C_1 M^\beta)\tau^{-p}.
\]
Pick \(M\) so large that \(c_0 M^s - C_1 M^\beta \ge 4D\) for some \(D>0\). Then
\[
C_1\Lambda^\beta - c_0\tau\Lambda^s \le -4D\tau^{-p}.
\]
The polynomial factor \(\tau^{-n/(4s)}\) can be absorbed by the exponential, i.e.
\[
\tau^{-n/(4s)} e^{-4D\tau^{-p}} \le C e^{-3D\tau^{-p}},\qquad 0<\tau\le T.
\]
Choosing \(M\) also large enough so that \(\frac12 M^s \ge 3D\), we obtain
\[
e^{-\frac{\tau}{2}\Lambda^s} = e^{-\frac12 M^s\tau^{-p}} \le e^{-3D\tau^{-p}}.
\]
Thus the coefficient of \(\|u(\cdot,t_1)\|_{L^2(\Omega)}\) in \eqref{eq:single-step-beta} is bounded by \(C e^{-3D\tau^{-p}}\).
Enlarging \(C\) if necessary, we rewrite the single-step estimate \eqref{eq:single-step-beta} as
\begin{equation}\label{eq:single-step-beta-final}
\|u(\cdot,t_2)\|_{L^2(\Omega)}
\le C_3 \exp\Bigl(\frac{C_4}{\tau^{p}}\Bigr) \int_{t_1}^{t_2} \sup_{x\in\omega}|u(x,t)|\,\diff t
+ e^{-D/\tau^{p}} \|u(\cdot,t_1)\|_{L^2(\Omega)},
\end{equation}
where \(C_3,C_4\) and \(D\) are now fixed positive constants.

Set \(l_m = T/2^m\) (\(m=0,1,\dots\)), so that \(\tau_m = l_m-l_{m+1} = T/2^{m+1}\).
Apply \eqref{eq:single-step-beta-final} with \(t_1=l_{m+1}\), \(t_2=l_m\):
\[
\|u(\cdot,l_m)\|_{L^2(\Omega)}
\le C_3 \exp\Bigl(\frac{C_4}{\tau_m^{p}}\Bigr) \int_{l_{m+1}}^{l_m} \sup_{\omega}|u|\,\diff t
+ e^{-D/\tau_m^{p}} \|u(\cdot,l_{m+1})\|_{L^2(\Omega)}.
\]
Introduce the weights \(a_m = \exp(-E/\tau_m^{p})\) with \(E>C_4\) to be chosen.
Multiplying by \(a_m\) gives
\[
a_m\|u(\cdot,l_m)\|_{L^2(\Omega)}
\le C_3 \exp\Bigl(\frac{C_4-E}{\tau_m^{p}}\Bigr) \int_{l_{m+1}}^{l_m} \sup_{\omega}|u(x,t)|\,\diff t
+ \exp\Bigl(-\frac{E+D}{\tau_m^{p}}\Bigr) \|u(\cdot,l_{m+1})\|_{L^2(\Omega)}.
\]
Because \(E>C_4\), \(\exp\bigl(\frac{C_4-E}{\tau_m^{p}}\bigr) \le 1\).
Since \(\tau_{m+1} = \tau_m/2\), we have \(a_{m+1} = \exp(-E 2^{p}/\tau_m^{p})\).
Choose \(D\) large enough so that \(E+D \ge E 2^{p}\). Then
\[
\exp\Bigl(-\frac{E+D}{\tau_m^{p}}\Bigr) \le a_{m+1}.
\]
Hence
\[
a_m\|u(\cdot,l_m)\|_{L^2(\Omega)} - a_{m+1}\|u(\cdot,l_{m+1})\|_{L^2(\Omega)}
\le C_3 \int_{l_{m+1}}^{l_m} \sup_{x\in\omega}|u(x,t)|\,\diff t.
\]
Summing over \(m=0,\dots,N\) gives
\[
a_0\|u(\cdot,T)\|_{L^2(\Omega)} - a_{N+1}\|u(\cdot,l_{N+1})\|_{L^2(\Omega)}
\le C_3 \int_{l_{N+1}}^T \sup_{x\in\omega}|u(x,t)|\,\diff t.
\]
As \(N\to\infty\), \(\tau_{N+1}=T/2^{N+2}\to0\), hence \(a_{N+1}=\exp(-E 2^{p}/\tau_{N+1}^p)\to0\), while \(\|u(\cdot,l_{N+1})\|_{L^2(\Omega)}\le\|u_0\|_{L^2(\Omega)}\) remains bounded. Moreover, the time interval \([l_{N+1},T]\) converges to \([0,T]\). Consequently, letting \(N\to\infty\) and dropping the non‑negative term \(a_{N+1}\|u(\cdot,l_{N+1})\|_{L^2(\Omega)}\) on the left‑hand side, we obtain
\[
a_0\|u(\cdot,T)\|_{L^2(\Omega)} \le C_3 \int_0^T \sup_{x\in\omega}|u(x,t)|\,\diff t .
\]
Since \(\tau_0 = T/2\), \(a_0 = \exp(-E 2^{p}/T^{p})\).
Recalling \(p = \beta/(s-\beta)\), we finally obtain
\[
\|u(\cdot,T)\|_{L^2(\Omega)} \le C \exp\Bigl(\frac{C}{T^{\beta/(s-\beta)}}\Bigr)
\int_0^T \sup_{x\in\omega}|u(x,t)|\,\diff t,
\]
which completes the proof.
\end{proof}

\begin{remark}
The condition \(\beta<s\) is sharp for the Lebeau--Robbiano iteration employed here.  
If \(\beta\ge s\), the spectral cost \(e^{C\Lambda^\beta}\) cannot be absorbed by the high-frequency decay \(e^{-c\tau\Lambda^s}\), and the present argument does not yield the desired observability inequality.
\end{remark}

\section{Logarithmic stability for recovering the initial data}\label{sec4}

\begin{proof}[\bf Proof of Theorem~\ref{thm:stability}]

Since a source term \(f\) is known, we can decompose the solution as \(u=u^0+u^f\), where \(u^0\) solves the homogeneous equation \eqref{eq:frac-heat} with the given initial data and \(u^f\) is the Duhamel contribution
\[
u^f(x,t)=\int_0^t e^{-(t-\tau)A^s}f(x,\tau)\,\diff\tau .
\]
Since \(u^f\) is known, the difference of two solutions reduces to the homogeneous case: \(u-\widetilde u = u^0-\widetilde u^0\).

Set \(w(x,t)=u^0(x,t)-\tilde u^0(x,t)\) and \(w_0(x)=w(x,0)\). Then \(w\) satisfies the homogeneous equation
\[
\partial_t w+A^s w=0,\quad w(0)=w_0,
\]
and the observation error is exactly \(\varepsilon=\int_0^T\sup_{x\in\omega}|w(x,t)|\,\diff t\).
From the a priori bounds we obtain
\[
\|A^r w_0\|_{L^2(\Omega)}\le M,\qquad
\|w_0\|_{L^2(\Omega)}\le \lambda_1^{-r}M .
\]

If \(\varepsilon=0\), the observability inequality \eqref{eq:obs-beta-for-stab} gives \(w(\cdot,T)=0\).
Since the semigroup \(e^{-TA^s}\) is injective on \(L^2(\Omega)\) (all eigenvalues satisfy \(e^{-T\lambda_k^s}>0\)), we conclude that \(w_0=0\), so the desired estimate holds trivially. Thus we assume \(\varepsilon>0\) in what follows.

Apply Theorem~\ref{thm:obs} to \(w\). The observability inequality \eqref{eq:obs-beta-for-stab} yields
\begin{equation}\label{eq:terminal-L2}
\|w(\cdot,T)\|_{L^2(\Omega)}\le C_T\varepsilon,
\qquad
C_T:=C_{\mathrm{obs}}\exp\Bigl(\frac{C_{\mathrm{obs}}}{T^{\beta/(s-\beta)}}\Bigr).
\end{equation}
The constant \(C_T\) depends on \(\Omega,\omega,s,\beta,T\) and on the constants appearing in the corresponding spectral inequality.

For an arbitrary \(\Lambda>0\), split \(w_0\) into low- and high-frequency components:
\[
w_0=\Pi_\Lambda w_0+(I-\Pi_\Lambda)w_0,
\]
where \(\Pi_\Lambda\) is the orthogonal projection onto the span of all eigenfunctions with \(\lambda_k\le\Lambda\). By the triangle inequality,
\begin{equation}\label{eq:decomp}
\|w_0\|_{L^2(\Omega)}\le\|\Pi_\Lambda w_0\|_{L^2(\Omega)}+\|(I-\Pi_\Lambda)w_0\|_{L^2(\Omega)}.
\end{equation}

Because the semigroup commutes with \(\Pi_\Lambda\),
\[
\Pi_\Lambda w(\cdot,T)=\Pi_\Lambda e^{-TA^s}w_0=e^{-TA^s}\Pi_\Lambda w_0.
\]
The restriction of \(e^{-TA^s}\) to the finite-dimensional range of \(\Pi_\Lambda\) is invertible, with inverse \(e^{TA^s}\). Hence
\[
\Pi_\Lambda w_0=e^{TA^s}\Pi_\Lambda w(x,T).
\]
Using the spectral expansion, for any \(g\in L^2(\Omega)\),
\[
\|e^{TA^s}\Pi_\Lambda g\|_{L^2(\Omega)}^2
=\sum_{\lambda_k\le\Lambda}e^{2T\lambda_k^s}|\langle g,\varphi_k\rangle|^2
\le e^{2T\Lambda^s}\sum_{\lambda_k\le\Lambda}|\langle g,\varphi_k\rangle|^2
= e^{2T\Lambda^s}\|\Pi_\Lambda g\|_{L^2(\Omega)}^2.
\]
Applying this with \(g=w(x,T)\) and using \eqref{eq:terminal-L2} gives
\begin{equation}\label{eq:low}
\|\Pi_\Lambda w_0\|_{L^2(\Omega)}\le e^{T\Lambda^s}\|\Pi_\Lambda w(\cdot,T)\|_{L^2(\Omega)}
\le C_T e^{T\Lambda^s}\varepsilon .
\end{equation}

The a priori regularity \(w_0\in\dom(A^r)\) implies
\[
\|(I-\Pi_\Lambda)w_0\|_{L^2(\Omega)}^2
=\sum_{\lambda_k>\Lambda}|\langle w_0,\varphi_k\rangle|^2
\le \Lambda^{-2r}\sum_{\lambda_k>\Lambda}\lambda_k^{2r}|\langle w_0,\varphi_k\rangle|^2
\le \Lambda^{-2r}M^2.
\]
Taking square roots yields
\begin{equation}\label{eq:high}
\|(I-\Pi_\Lambda)w_0\|_{L^2(\Omega)}\le M\Lambda^{-r}.
\end{equation}

Insert \eqref{eq:low} and \eqref{eq:high} into \eqref{eq:decomp} to obtain, for every \(\Lambda>0\),
\begin{equation}\label{eq:before-opt}
\|w_0\|_{L^2(\Omega)}\le C_T e^{T\Lambda^s}\varepsilon+M\Lambda^{-r}.
\end{equation}
The first term grows with \(\Lambda\) due to backward amplification of low frequencies, while the second term decays because higher modes are controlled by the a priori bound. We now choose \(\Lambda\) to balance these competing effects.

Assume temporarily that \(0<\varepsilon<M/(eC_T)\). Then \(\log\frac{M}{C_T\varepsilon}>1\). Set
\[
\Lambda=\Bigl(\frac{1}{2T}\log\frac{M}{C_T\varepsilon}\Bigr)^{1/s}.
\]
Consequently \(T\Lambda^s=\frac12\log\frac{M}{C_T\varepsilon}\) and \(e^{T\Lambda^s}=\sqrt{M/(C_T\varepsilon)}\).
Substituting these relations into \eqref{eq:before-opt} gives
\[
\|w_0\|_{L^2(\Omega)}\le \sqrt{C_T M\varepsilon}+(2T)^{r/s}M\Bigl(\log\frac{M}{C_T\varepsilon}\Bigr)^{-r/s}.
\]
To absorb the first term, write \(\eta:=\log\frac{M}{C_T\varepsilon}>1\). Then
\[
\sqrt{C_T M\varepsilon}=M e^{-\eta/2}\le C_{r,s,T} M \eta^{-r/s}
\]
for a suitable constant \(C_{r,s,T}\), because the exponential decays faster than any negative power of \(\eta\) as \(\eta\to\infty\).
Thus
\begin{equation}\label{eq:log-with-CT}
\|w_0\|_{L^2(\Omega)}\le C_1 M\Bigl(\log\frac{M}{C_T\varepsilon}\Bigr)^{-r/s},
\qquad 0<\varepsilon<\frac{M}{eC_T},
\end{equation}
with a constant \(C_1\) depending on \(r,s,T\) and \(C_T\).

For \(\varepsilon\) in the considered range, since \(\eta = \log\frac{M}{C_T\varepsilon}>1\), then \(M/\varepsilon = C_T e^{\eta}\), and
\[
\log\Bigl(e+\frac{M}{\varepsilon}\Bigr)
= \log(e + C_T e^{\eta})
= \eta + \log\Bigl(C_T + e^{1-\eta}\Bigr).
\]
Since \(\eta>1\), we have \(0< e^{1-\eta}\le 1\), hence
\[
\log C_T < \log(C_T + e^{1-\eta}) \le \log(C_T+1).
\]
Therefore
\[
\eta + \log C_T < \log\Bigl(e+\frac{M}{\varepsilon}\Bigr) \le \eta + \log(C_T+1).
\]
Moreover, there exist constants \(c_1, c_2 > 0\) depending only on \(C_T\) such that
\[
c_1 \eta \;\le\; \eta + \log C_T, \quad \eta + \log(C_T+1) \;\le\; c_2 \eta .
\]
Indeed, one may take \(c_1 = 1 - |\log C_T|\) if \(|\log C_T| < 1\), and a smaller positive constant otherwise; \(c_2 = 1 + \log(C_T+1)\) always works.
Thus
\[
c_1 \log\frac{M}{C_T\varepsilon}
\;\le\;
\log\Bigl(e+\frac{M}{\varepsilon}\Bigr)
\;\le\;
c_2 \log\frac{M}{C_T\varepsilon}.
\]
Raising to the power \(-r/s\) reverses the inequalities:
\[
c_2^{-r/s}\Bigl(\log\frac{M}{C_T\varepsilon}\Bigr)^{-r/s}
\;\le\;
\Bigl[\log\Bigl(e+\frac{M}{\varepsilon}\Bigr)\Bigr]^{-r/s}
\;\le\;
c_1^{-r/s}\Bigl(\log\frac{M}{C_T\varepsilon}\Bigr)^{-r/s}.
\]
Combining this with \eqref{eq:log-with-CT} yields, for \(0<\varepsilon<M/(eC_T)\),
\[
\|w_0\|_{L^2(\Omega)} \le C_2 M \Bigl[\log\Bigl(e+\frac{M}{\varepsilon}\Bigr)\Bigr]^{-r/s},
\]
where \(C_2 = C_1 c_2^{r/s}\) depends on \(r,s,T,C_T\) and on the structural constants of the observation set.

It remains to consider the case \(\varepsilon\ge M/(eC_T)\). In this regime we rely directly on the a priori estimate for \(w_0\). Since \(\lambda_1>0\) and \(\|A^r w_0\|_{L^2(\Omega)}\le M\),
\[
\|w_0\|_{L^2(\Omega)}\le \lambda_1^{-r}\|A^r w_0\|_{L^2(\Omega)}\le \lambda_1^{-r}M.
\]
On the other hand, the condition \(\varepsilon\ge M/(eC_T)\) implies \(\frac{M}{\varepsilon}\le eC_T\), so that
\[
\log\Bigl(e+\frac{M}{\varepsilon}\Bigr)\le \log(e+eC_T).
\]
Hence the logarithmic factor \(\bigl[\log(e+M/\varepsilon)\bigr]^{-r/s}\) is bounded from below by the positive constant \(\bigl[\log(e+eC_T)\bigr]^{-r/s}\). Consequently,
\[
\lambda_1^{-r}M \le C_3 M\Bigl[\log\Bigl(e+\frac{M}{\varepsilon}\Bigr)\Bigr]^{-r/s}
\]
with \(C_3 = \lambda_1^{-r}\bigl[\log(e+eC_T)\bigr]^{r/s}\).

Choosing the final constant as \(C = \max\{C_2, C_3\}\), we obtain the desired inequality
\[
\|w_0\|_{L^2(\Omega)} \le C M\Bigl[\log\Bigl(e+\frac{M}{\varepsilon}\Bigr)\Bigr]^{-r/s},
\qquad\text{for all }\varepsilon>0.
\]
Together with the trivial case \(\varepsilon=0\) treated at the beginning, the proof is complete.
\end{proof}

\begin{remark}
Theorem~\ref{thm:stability} shows that the recovery of the initial data from thin-set observations is logarithmically stable under an a priori \(A^r\)-bound. The logarithmic rate \([\log(e+M/\varepsilon)]^{-r/s}\) arises from balancing the backward amplification \(e^{T\Lambda^s}\) of low frequencies against the high-frequency tail \(M\Lambda^{-r}\) controlled by the a priori smoothness. The exponent \(\beta\) from the spectral inequality does not affect this logarithmic rate; it influences only the observability constant \(C_T\), whose short-time behaviour is
\[
C_T \sim \exp\Bigl(\frac{C}{T^{\beta/(s-\beta)}}\Bigr).
\]
Thus the stability rate in \(\varepsilon\) remains logarithmic with exponent \(r/s\), while the multiplicative constant deteriorates as \(\beta\) approaches \(s\).     
\end{remark}

\section{Numerical experiments}\label{sec5}

The forward model is truncated to a finite number of eigenmodes, synthetic data are corrupted by additive Gaussian noise, and the inversion is performed via a regularised least‑squares method with a spectral penalty on the higher eigenmodes. The regularisation parameter is chosen by the discrepancy principle. We first describe the unified numerical algorithm and then provide a conditional convergence analysis based on the logarithmic stability estimate established in Theorem~\ref{thm:stability}.

\subsection{Numerical algorithm}\label{sec:algorithm}

Let \(\Omega=(0,1)^n\) with \(n=1\) or \(n=2\). The Dirichlet Laplacian \(A=-\Delta\) admits the orthonormal eigenfunctions
\[
\varphi_{\mathbf{k}}(x)=(\sqrt{2})^{\,n}\prod_{j=1}^{n}\sin(k_j\pi x_j),\qquad
\mathbf{k}=(k_1,\dots,k_n)\in\mathbb N^n,
\]
with eigenvalues \(\lambda_{\mathbf{k}}=\pi^2\sum_{j=1}^{n}k_j^2\). The fractional Laplacian acts as \(A^s\varphi_{\mathbf{k}}=\lambda_{\mathbf{k}}^s\varphi_{\mathbf{k}}\).

We consider the nonhomogeneous equation
\[
\partial_t u+A^s u=f(x,t),\quad u(0)=u_0,\quad u|_{\partial\Omega}=0,
\]
with a known source term \(f\). Setting \(a_{\mathbf{k}}=\langle u_0,\varphi_{\mathbf{k}}\rangle\) and \(f_{\mathbf{k}}(t)=\langle f(\cdot,t),\varphi_{\mathbf{k}}\rangle\), Duhamel's formula gives the decomposition \(u=u^0+u^f\), where
\[
u^0(x,t)=\sum_{\mathbf{k}}a_{\mathbf{k}}e^{-t\lambda_{\mathbf{k}}^s}\varphi_{\mathbf{k}}(x),\qquad
u^f(x,t)=\sum_{\mathbf{k}}\varphi_{\mathbf{k}}(x)\int_0^t e^{-(t-\tau)\lambda_{\mathbf{k}}^s}f_{\mathbf{k}}(\tau)\,\diff\tau .
\]
The term \(u^f\) is known, while \(u^0\) contains the unknown initial data.

For numerical computation, we truncate the expansion to the finite index set
\[
\mathcal I_N=\{\mathbf{k}\in\mathbb N^n:1\le k_j\le N,\;j=1,\dots,n\},
\]
and denote by \(\mathbf a\in\mathbb R^{\mathcal I_N}\) the vector of unknown coefficients. The approximate solution is
\[
u_N(x,t)=\sum_{\mathbf{k}\in\mathcal I_N} a_{\mathbf{k}} e^{-t\lambda_{\mathbf{k}}^s}\varphi_{\mathbf{k}}(x)
+\sum_{\mathbf{k}\in\mathcal I_N}\varphi_{\mathbf{k}}(x)\int_0^t e^{-(t-\tau)\lambda_{\mathbf{k}}^s}f_{\mathbf{k}}(\tau)\,\diff\tau .
\]
We write \(u_N = u_N^0 + u_N^f\) for the homogeneous and source parts. The exact source response \(u^f\) and its truncated version \(u_N^f\) differ by a source truncation error, which will be analysed in the convergence study.

To evaluate the Fourier coefficients and the Duhamel integrals, we employ a uniform spatial quadrature grid. In each coordinate direction we take \(N_{\rm int}\) equally spaced interior nodes
\[
x_j = \frac{j}{N_{\rm int}+1},\qquad j=1,\dots,N_{\rm int},
\]
and in two dimensions we use the Cartesian product of such one‑dimensional grids. 

Choose a finite observation set \(\omega=\{x^{(1)},\dots,x^{(M_x)}\}\subset\Omega\) and discrete times \(t_i=i\Delta t\), \(i=1,\dots,N_t\), with \(\Delta t=T/N_t\). Define the homogeneous forward matrix \(\mathbf K\in\mathbb R^{M_xN_t\times|\mathcal I_N|}\) by
\[
K_{(i,j),\mathbf{k}}=e^{-t_i\lambda_{\mathbf{k}}^s}\varphi_{\mathbf{k}}(x^{(j)}),
\]
and the truncated source data vector \(\mathbf d^f\) by
\[
(\mathbf d^f)_{(i,j)}=\sum_{\mathbf{k}\in\mathcal I_N}\varphi_{\mathbf{k}}(x^{(j)})\int_0^{t_i} e^{-(t_i-\tau)\lambda_{\mathbf{k}}^s}f_{\mathbf{k}}(\tau)\,\diff\tau .
\]

Let \(\mathbf a^{\rm true}\) be the coefficient vector of the truncated initial data \(u_{0,N}=\Pi_{\Lambda_N}u_0\). The exact, noise-free data at the observation points are
\[
\mathbf d_{\rm clean} = \mathbf K\mathbf a^{\rm true} + \mathbf d^f_{\rm exact},
\]
where \(\mathbf d^f_{\rm exact}\) denotes the contribution of the full (untruncated) source response at the sensors. In practice we only have access to the truncated approximation \(\mathbf d^f\). The source truncation error is then defined as
\[
\mathbf e^f = \mathbf d^f_{\rm exact} - \mathbf d^f .
\]

To simulate measurement noise, we add independent Gaussian noise \(\boldsymbol\eta\) with standard deviation
\[
\sigma = \delta\,\max_m |(\mathbf d_{\rm clean})_m|,\qquad m=1,2,\dots,M_x N_t,
\]
where \(\delta\in(0,1)\) is the prescribed relative noise level. The noisy data vector actually used for inversion is therefore
\[
\mathbf d = \mathbf d_{\rm clean} + \boldsymbol\eta = \mathbf K\mathbf a^{\rm true} + \mathbf d^f + \mathbf e^f + \boldsymbol\eta .
\]

Since the source term \(f\) is known, the truncated contribution \(\mathbf d^f\) can be computed and subtracted, yielding the corrected data
\[
\mathbf d^{\,0} = \mathbf d - \mathbf d^f = \mathbf K\mathbf a^{\rm true} + \mathbf e^f + \boldsymbol\eta .
\]
After this correction, the inverse problem is reduced to the homogeneous case, but now with an additional deterministic error \(\mathbf e^f\) that accounts for the finite truncation of the source term. This error is further analysed in the convergence study below.

We recover \(\mathbf a\) by solving
\begin{equation}\label{eq:reg-source}
\hat{\mathbf a}=\arg\min_{\mathbf a}\bigl(\|\mathbf K\mathbf a-\mathbf d^{\,0}\|_2^2+\gamma\|L\mathbf a\|_2^2\bigr),
\end{equation}
where \(\gamma>0\) and \(L=\operatorname{diag}(\lambda_{\mathbf{k}}^{r_{\rm reg}/2})_{\mathbf{k}\in\mathcal I_N}\) with \(r_{\rm reg}>0\). The minimisation problem \eqref{eq:reg-source} is quadratic and strictly convex for \(\gamma>0\), since \(L^{\mathsf T}L\) is positive definite on the finite-dimensional coefficient space. Its unique minimiser satisfies the normal equations
\[
(\mathbf K^{\mathsf T}\mathbf K+\gamma L^{\mathsf T}L)\hat{\mathbf a}=\mathbf K^{\mathsf T}\mathbf d^{\,0}.
\]
The parameter \(\gamma\) is chosen by the discrepancy principle applied to the corrected data:
\[
\|\mathbf K\hat{\mathbf a}-\mathbf d^{\,0}\|_2\approx\tau_{\rm disc}\|\boldsymbol\eta\|_2,\qquad \tau_{\rm disc}>1.
\]

The reconstructed initial condition is \(u_0^{\rm rec}(x)=\sum_{\mathbf{k}\in\mathcal I_N}\hat a_{\mathbf{k}}\varphi_{\mathbf{k}}(x)\). Its accuracy is measured by the relative \(L^2\) error
\[
{\rm err}=\frac{\|u_0^{\rm rec}-u_0^{\rm true}\|_{L^2(\Omega_{\rm fine})}}{\|u_0^{\rm true}\|_{L^2(\Omega_{\rm fine})}}
\]
evaluated on a fine grid \(\Omega_{\rm fine}\).

\subsection{Convergence analysis}\label{sec:convergence}

We now provide a conditional convergence analysis for the known‑source reconstruction. The analysis accounts for the truncation errors in both the initial data and the source term. Assume that \(\omega\) satisfies a spectral inequality
\begin{equation}\label{eq:unified-spectral-numerics}
\|\Pi_\Lambda u\|_{L^2(\Omega)}\le C e^{C\Lambda^\beta}\sup_{x\in\omega}|\Pi_\Lambda u(x)|,\qquad \beta\in(0,s).
\end{equation}
For general thin sets \(\beta=1/2\); for the special zero‑dimensional sets one may choose any \(\beta\in(1/2,s)\). Let the true initial data satisfy \(u_0\in\dom(A^r)\) with \(\|A^r u_0\|_{L^2(\Omega)}\le M\). Denote by
\[
\Lambda_N=\max_{\mathbf{k}\in\mathcal I_N}\lambda_{\mathbf{k}}=\pi^2 n N^2\sim N^2
\]
the maximal retained eigenvalue. The truncated initial data is \(u_{0,N}=\Pi_{\Lambda_N}u_0=\sum_{\mathbf{k}\in\mathcal I_N}a_{\mathbf{k}}^{\rm true}\varphi_{\mathbf{k}}\), and the reconstructed initial data is \(u_0^{\rm rec}=\sum_{\mathbf{k}\in\mathcal I_N}\hat a_{\mathbf{k}}\varphi_{\mathbf{k}}\).

We split the total error as
\begin{equation}\label{eq:error-decomp}
\|u_0-u_0^{\rm rec}\|_{L^2(\Omega)}
\le \underbrace{\|u_0-u_{0,N}\|_{L^2(\Omega)}}_{\text{truncation error for }u_0}
+ \underbrace{\|u_{0,N}-u_0^{\rm rec}\|_{L^2(\Omega)}}_{\text{inversion error}} .
\end{equation}

Using the spectral definition of \(A^r\),
\[
\|u_0-u_{0,N}\|_{L^2(\Omega)}^2=\sum_{\lambda_{\mathbf{k}}>\Lambda_N}|\langle u_0,\varphi_{\mathbf{k}}\rangle|^2
\le \Lambda_N^{-2r}\|A^r u_0\|_{L^2(\Omega)}^2\le M^2\Lambda_N^{-2r},
\]
hence
\begin{equation}\label{eq:trunc-error}
\|u_0-u_{0,N}\|_{L^2(\Omega)}\le M\Lambda_N^{-r}\sim MN^{-2r}.
\end{equation}

Suppose additionally that the source term satisfies \(f\in L^1(0,T;\dom(A^\rho))\) for some \(\rho>0\), so that its Fourier coefficients decay as \(|f_{\mathbf{k}}(t)|\lesssim \lambda_{\mathbf{k}}^{-\rho}\) uniformly in time. The difference between the exact and truncated source data is \(\mathbf e^f = \mathbf d^f_{\rm exact} - \mathbf d^f\). By the smoothness of \(f\), the high‑frequency components of the source response satisfy
\[
\|u^f(\cdot,t)-u_N^f(\cdot,t)\|_{L^\infty(\Omega)}\le C t^{-n/(4s)}\|A^\rho f(\cdot,t)\|_{L^2(\Omega)}\Lambda_N^{-\rho}
\]
for \(t>0\). Evaluating at the discrete times and summing over the observation points yields
\[
\|\mathbf e^f\|_2 \le C_{\rm grid}^f \Lambda_N^{-\rho},
\]
where the constant depends on the time grid and the \(L^1_t\) norm of \(\|A^\rho f(\cdot,t)\|_{L^2(\Omega)}\). Thus the source truncation error is of the same order as the initial truncation error when \(\rho=r\).

Let \(\mathbf a^{\rm true}\) be the coefficient vector of \(u_{0,N}\). The corrected data can be written as
\[
\mathbf d^{\,0}=\mathbf K\mathbf a^{\rm true}+\mathbf r_N+\mathbf e^f+\boldsymbol\eta,
\]
where \(\mathbf r_N\) is the modelling error due to replacing \(u_0\) by \(u_{0,N}\) in the homogeneous part, i.e.\ \((\mathbf r_N)_{(i,j)}=[e^{-t_iA^s}(u_0-u_{0,N})](x^{(j)})\). By the discrepancy principle,
\[
\|\mathbf K\hat{\mathbf a}-\mathbf d^{\,0}\|_2\le C_\eta\|\boldsymbol\eta\|_2,
\]
and consequently
\[
\|\mathbf K(\hat{\mathbf a}-\mathbf a^{\rm true})\|_2\le (C_\eta+1)\|\boldsymbol\eta\|_2+\|\mathbf r_N\|_2+\|\mathbf e^f\|_2 .
\]

To link this discrete residual to the continuous observation error
\[
\varepsilon_N=\int_0^T\sup_{x\in\omega}|u_N^0(x,t)- u^{0,\rm rec}(x,t)|\,\diff t,
\]
where \(u_N^0(t)=e^{-tA^s}u_{0,N}\) and \( u^{0,\rm rec}(t)=e^{-tA^s}u_0^{\rm rec}\), we assume a finite‑dimensional sampling consistency condition: there exists \(C_{\rm samp}>0\) such that for every \(z_0\in X_N=\operatorname{span}\{\varphi_{\mathbf{k}}:\mathbf{k}\in\mathcal I_N\}\),
\begin{equation}\label{eq:sampling-assumption}
\int_0^T\sup_{x\in\omega}|e^{-tA^s}z_0(x)|\,\diff t
\le C_{\rm samp}\Bigl(\sum_{i=1}^{N_t}\sum_{j=1}^{M}|e^{-t_iA^s}z_0(x^{(j)})|^2\Bigr)^{1/2}.
\end{equation}
Applying it to \(z_0=u_{0,N}-u_0^{\rm rec}\) gives
\[
\varepsilon_N\le C_{\rm samp}\|\mathbf K(\mathbf a^{\rm true}-\hat{\mathbf a})\|_2\le C\bigl(\|\boldsymbol\eta\|_2+\|\mathbf r_N\|_2+\|\mathbf e^f\|_2\bigr).
\]

The modelling error \(\mathbf r_N\) is bounded via the semigroup smoothing: for \(t_i>0\),
\[
|e^{-t_iA^s}(u_0-u_{0,N})(x^{(j)})|\le C t_i^{-n/(4s)}\|u_0-u_{0,N}\|_{L^2(\Omega)},
\]
so that \(\|\mathbf r_N\|_2\le C_{\rm grid}\|u_0-u_{0,N}\|_{L^2(\Omega)}\) with \(C_{\rm grid}\) finite because \(t_i\ge\Delta t>0\). Combining these bounds with \(\|\boldsymbol\eta\|_2\le C_{\rm noise}\delta\), we obtain
\begin{equation}\label{eq:effective-eps}
\varepsilon_N\le C_{\rm data}\bigl(\delta+\Lambda_N^{-r}+\Lambda_N^{-\rho}\bigr).
\end{equation}

We now apply Theorem~\ref{thm:stability} to the homogeneous components \(u_{0,N}\) and \(u_0^{\rm rec}\). Both belong to \(X_N\subset\dom(A^r)\). We assume that \(\|A^ru_0^{\rm rec}\|_{L^2(\Omega)}\le M_{\rm rec}\). Together with \(\|A^r u_{0,N}\|_{L^2(\Omega)}\le M\), we set \(\widetilde M=M+M_{\rm rec}\). Theorem~\ref{thm:stability} yields
\[
\|u_{0,N}-u_0^{\rm rec}\|_{L^2(\Omega)}\le C\widetilde M\Bigl[\log\Bigl(e+\frac{\widetilde M}{\varepsilon_N}\Bigr)\Bigr]^{-r/s}.
\]
Inserting \eqref{eq:effective-eps} and absorbing constants gives
\begin{equation}\label{eq:inv-error-final}
\|u_{0,N}-u_0^{\rm rec}\|_{L^2(\Omega)}\le C\widetilde M\Bigl[\log\Bigl(e+\frac{\widetilde M}{\delta+\Lambda_N^{-r}+\Lambda_N^{-\rho}}\Bigr)\Bigr]^{-r/s}.
\end{equation}

Combining the error decomposition \eqref{eq:error-decomp}, the truncation estimate \eqref{eq:trunc-error}, and the inversion estimate \eqref{eq:inv-error-final}, and using \(\Lambda_N\sim N^2\), we obtain
\[
\|u_0-u_0^{\rm rec}\|_{L^2(\Omega)}
\le C_1MN^{-2r}+C_{2}\widetilde M\Bigl[\log\Bigl(e+\frac{\widetilde M}{\delta+N^{-2r}+N^{-2\rho}}\Bigr)\Bigr]^{-r/s}.
\]

If the source has the same smoothness as the initial data (\(\rho=r\)), the dominant truncation term remains \(N^{-2r}\), and the additional term \(N^{-2\rho}\) inside the logarithm does not alter the asymptotic rates. Letting \(N\to\infty\) with \(\delta>0\) fixed yields
\[
\limsup_{N\to\infty}\|u_0-u_0^{\rm rec}\|_{L^2(\Omega)}\le C_{2}\widetilde M\Bigl[\log\Bigl(e+\frac{\widetilde M}{\delta}\Bigr)\Bigr]^{-r/s}.
\]
As \(\delta\to0\), the error decays at the logarithmic rate \([\log(1/\delta)]^{-r/s}\). We summarise the discussion in a conditional convergence theorem.

\begin{theorem}[Conditional convergence]\label{thm:conv}
Let the assumptions of the convergence analysis hold: the true initial data satisfies \(\|A^r u_0\|_{L^2(\Omega)}\le M\); the source term satisfies \(f\in L^1(0,T;\dom(A^\rho))\) for some \(\rho>0\); the observation set satisfies the unified spectral inequality \eqref{eq:unified-spectral-numerics} with exponent \(\beta\in(0,s)\); the finite‑dimensional sampling consistency condition \eqref{eq:sampling-assumption} is fulfilled; and the reconstructed initial data satisfies the a priori bound \(\|A^ru_0^{\rm rec}\|_{L^2(\Omega)}\le M_{\rm rec}\). Let \(u_0^{\rm rec}\) be the reconstruction obtained from the regularised least‑squares problem \eqref{eq:reg-source} with corrected data \(\mathbf d^{\,0}=\mathbf d-\mathbf d^f\) and with the regularisation parameter chosen by the discrepancy principle. Then \(u_0^{\rm rec}\) satisfies the error estimate 
\[
\|u_0-u_0^{\rm rec}\|_{L^2(\Omega)}
\le C_1MN^{-2r}+C_{2}\widetilde M\Bigl[\log\Bigl(e+\frac{\widetilde M}{\delta+N^{-2r}+N^{-2\rho}}\Bigr)\Bigr]^{-r/s},
\]
where \(\widetilde M=M+M_{\rm rec}\), \(\delta\) is the relative noise level, and the constants \(C_1,C_{2}\) depend on \(\Omega,\omega,s,r,\rho,T,\beta\) and on the sampling grid, but not on \(\delta\). In particular, for fixed \(\delta>0\),
\[
\limsup_{N\to\infty}\|u_0-u_0^{\rm rec}\|_{L^2(\Omega)}
\le C_{2}\widetilde M\Bigl[\log\Bigl(e+\frac{\widetilde M}{\delta}\Bigr)\Bigr]^{-r/s},
\]
and as \(\delta\to0\) the error decays at the logarithmic rate \([\log(1/\delta)]^{-r/s}\).
\end{theorem}

\subsection{Numerical reconstruction}

In the subsection, we present numerical reconstructions of initial data for the fractional heat equation on the interval \(\Omega=(0,1)\) (one‑dimensional) and on the square \(\Omega=(0,1)^2\) (two‑dimensional). In all experiments the observation points are chosen as finite subsets of the special zero‑dimensional observation sets described in Lemma~\ref{lem:spectral-special}. The theoretical spectral inequalities for such sets are proved for the full infinite constructions (for example, the whole sequence \(\{i^{-\alpha}\}_{i\ge1}\)); in the computations we naturally use only finitely many sensors, so the finite sets below should be regarded as practical approximations of the corresponding theoretical observation sets.

\subsubsection{One-dimensional }\label{sec:num1d}

We set \(\Omega=(0,1)\), truncate the expansion to \(N=30\) modes, and take \(s=0.8\), \(T=1\), \(N_{\rm int}=100\) and \(N_t=50\). According to Lemma~\ref{lem:spectral-special}, a one-dimensional special observation set may be taken as a rapidly accumulating sequence or an algebraic irrational point. We test two finite samplings:
\begin{itemize}
\item \textbf{Configuration A} (pure accumulation sequence): \(\omega= \bigl\{1,\frac12,\frac13,\frac14,\frac15,\frac16\bigr\}\);
\item \textbf{Configuration B} (single algebraic irrational point): \(\omega= \bigl\{\sqrt{2}-1\bigr\}\).
\end{itemize}
Both are finite subsets of the theoretically admissible set \(\omega_\alpha\cup\{x_0\}\) with \(\alpha=1\) and \(x_0=\sqrt{2}-1\). Configuration A samples the rapidly accumulating harmonic sequence, while Configuration B consists only of the algebraic irrational point.

The known source term is chosen as \(f(x,t)=\exp(-t)\,x\,(1-x)\). We form the truncated source data \(\mathbf d^f\) and the corrected data \(\mathbf d^{\,0}=\mathbf d-\mathbf d^f\) as described in the algorithm. The relative noise level is \(\delta=1\%\), and the regularisation uses \(r_{\rm reg}=2\) with \(L=\operatorname{diag}(\lambda_k)\), \(\lambda_k=(k\pi)^2\), and the regularisation parameter is selected by the discrepancy principle applied to the corrected data.

Two true initial conditions are tested:
\begin{example}\label{ex1}
Sine mode: \(u_0(x)=\sin(\pi x)\).
\end{example}
\begin{example}\label{ex2}
Gaussian: \(u_0(x)=\exp(-10(x-0.5)^2)\).
\end{example}

\begin{figure}[h!]\centering
\subfigure[\(\omega= \bigl\{1,\frac12,\frac13,\frac14,\frac15,\frac16\bigr\}\)]{\includegraphics[width=0.46\textwidth]{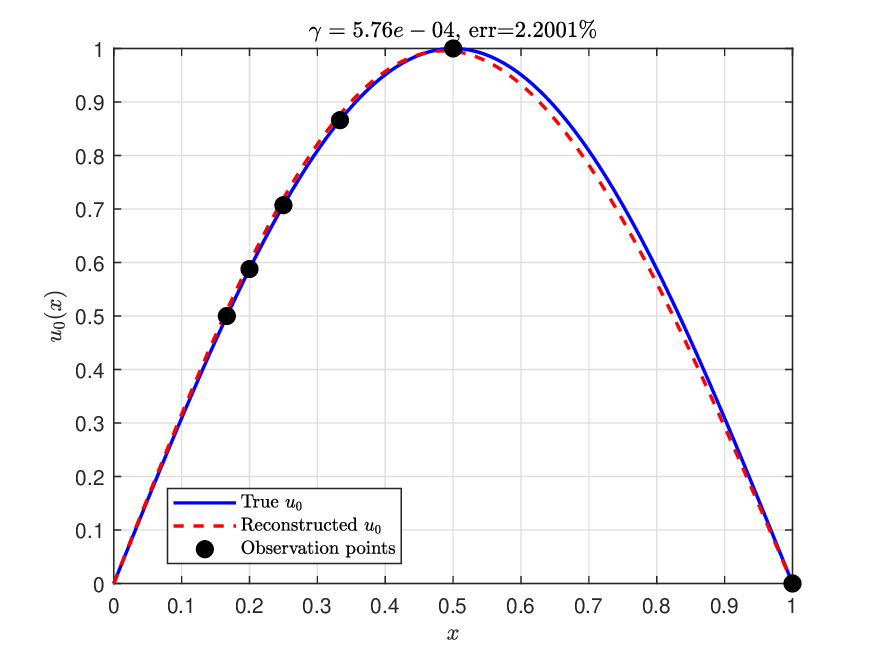}}
\subfigure[\(\omega=\{\sqrt{2}-1\}\)]{\includegraphics[width=0.46\textwidth]{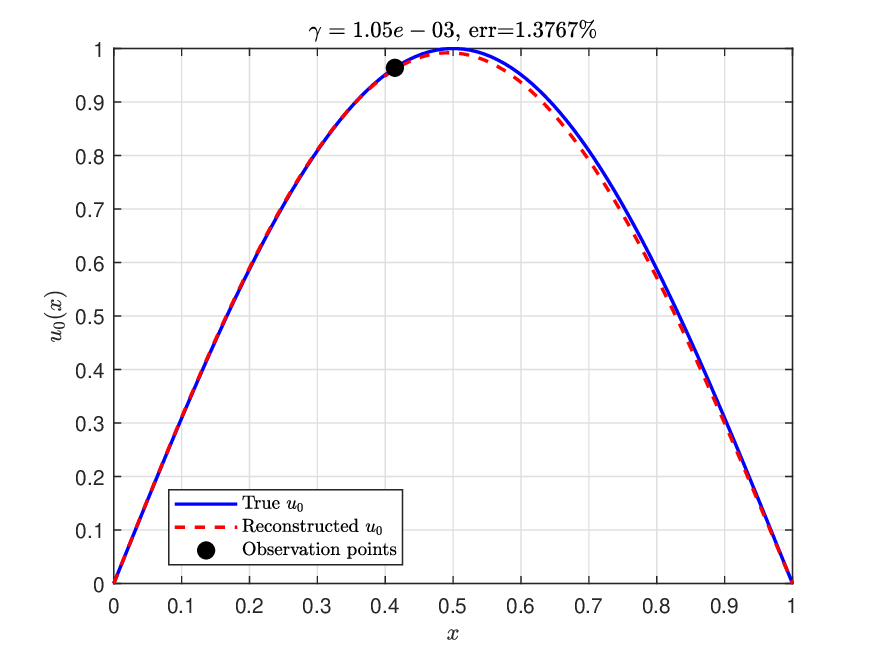}}
\caption{Comparison of the true solution and the reconstructed solution for Example \ref{ex1}.}\label{fig1}
\end{figure}

\begin{figure}[h!]\centering
\subfigure[\(\omega= \bigl\{1,\frac12,\frac13,\frac14,\frac15,\frac16\bigr\}\)]
{\includegraphics[width=0.46\textwidth]{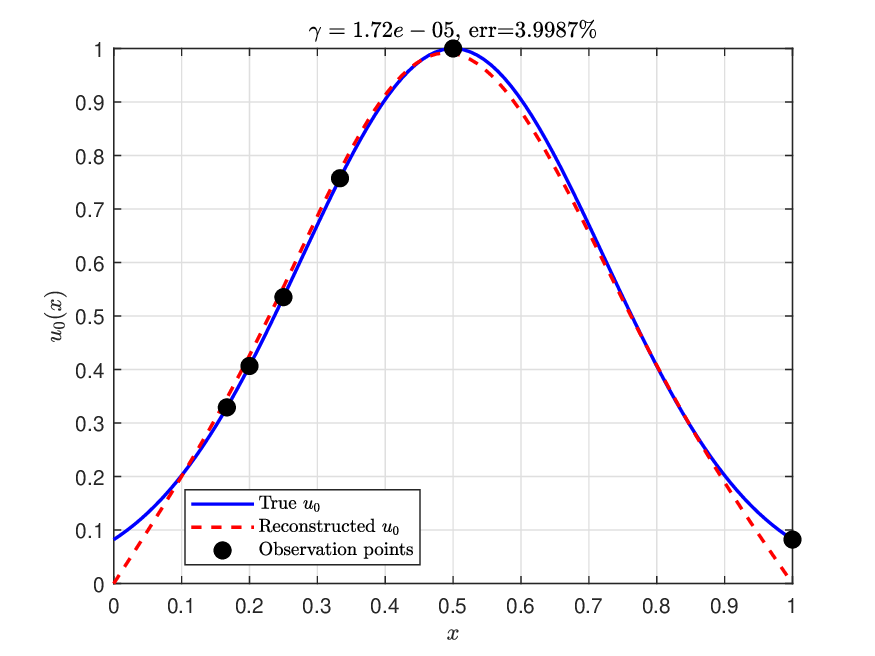}}
\subfigure[\(\omega=\{\sqrt{2}-1\}\)]{\includegraphics[width=0.46\textwidth]{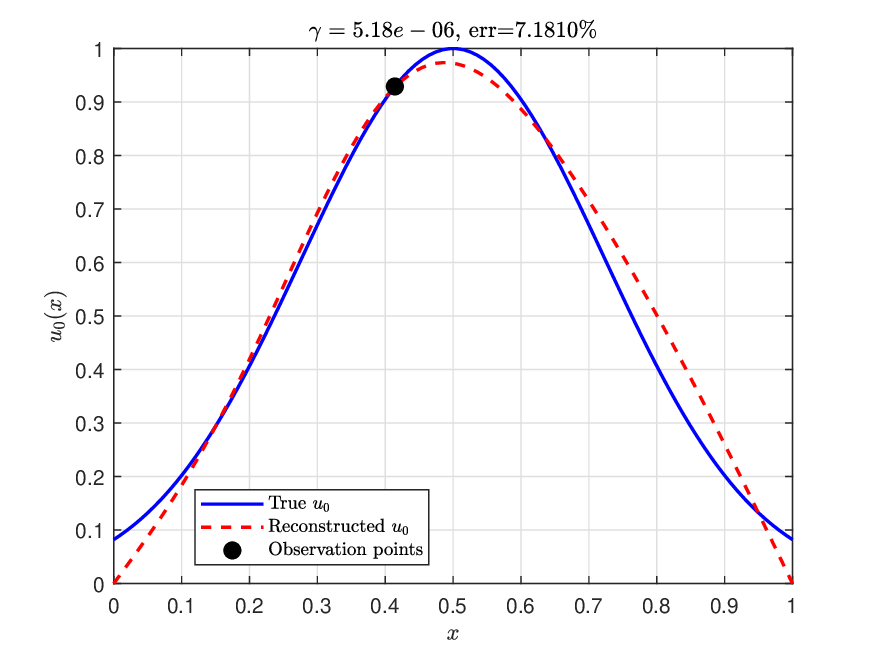}}
\caption{Comparison of the true solution and the reconstructed solution for Example \ref{ex2}.}
\label{fig2}
\end{figure}

The reconstruction results for the one-dimensional case are shown in Figures~\ref{fig1}--\ref{fig2}. Both configurations yield accurate recovery of the initial data. 
Configuration~A, which uses five rapidly accumulating points, produces a relative error of \(2.2001\%\) for the Sine mode and \(3.9987\%\) for the Gaussian profile. 
Configuration~B, which consists of the single algebraic irrational point \(\sqrt{2}-1\), gives errors of \(1.3767\%\) and \(7.1810\%\) respectively. 
These results confirm that the arithmetic nature of the observation point, as required by the Remez-type inequality underlying Lemma~\ref{lem:spectral-special}, plays a crucial role in the inversion.

A somewhat counterintuitive phenomenon emerges when comparing the two configurations. 
Although Configuration~A provides more spatial information, its advantage over the single-point observation in Configuration~B is rather modest. 
Indeed, for the Gaussian initial condition, the error of Configuration~B remains within a few percent despite using only one sensor. 
We suspect that the accumulation sequence, while increasing the amount of measured data, also introduces additional noise-contaminated channels that can slightly degrade the effective signal-to-noise ratio per degree of freedom, thereby partially offsetting the benefit of richer spatial sampling.

Compared with classical observations on nonempty open subsets, both configurations require dramatically fewer measurement points while still achieving high reconstruction accuracy. 
This is of considerable practical interest, as it demonstrates that accurate initial data recovery for fractional diffusion processes is feasible with extremely sparse sensor placements, potentially leading to significant cost reductions in applications where dense spatial measurements are impractical or expensive.

\subsubsection{Two-dimensional }\label{sec:num2d}

We take \(\Omega=(0,1)^2\), \(N=12\) modes per direction (144 unknowns), \(s=0.8\), \(T=1\), \(N_{\rm int}=30\) and \(N_t=50\). Lemma~\ref{lem:spectral-special} provides two admissible constructions in two dimensions: the Cartesian product \(\omega_\alpha\times\omega_\alpha\) and the mixed product \(\{x_0\}\times\omega_\alpha\). We test finite samplings of both. Let
\[
\omega_1= \Bigl\{1,\frac12,\,\frac13,\,\frac14,\,\frac15,\,\frac16\Bigr\},\qquad x_0={\sqrt{2}-1}.
\]
The two configurations are:
\begin{itemize}
\item \textbf{Configuration C} (Cartesian product): \(\omega= \omega_1\times\omega_1\) with \(|\omega|=36\) points.
\item \textbf{Configuration D} (mixed product): \(\omega= \{x_0\}\times\omega_1\) with \(|\omega|=6\) points.
\end{itemize}

The source term is chosen similarly as \(f(x,y,t)=\exp(-t)\,x(1-x)\,y(1-y)\). The truncated source data \(\mathbf d^f\) are computed, and the inversion is performed on the corrected data \(\mathbf d^{\,0}\). The relative noise level is \(\delta=1\%\), and the regularisation uses \(r_{\rm reg}=2\) with \(L=\operatorname{diag}(\lambda_{k_1,k_2})\), \(\lambda_{k_1,k_2}=\pi^2(k_1^2+k_2^2)\).

Two initial conditions are tested:

\begin{example}\label{ex3}
Sine mode: \(u_0(x,y)=\sin(\pi x)\sin(2\pi y)\).
\end{example}

\begin{example}\label{ex4}
Gaussian: \(u_0(x,y)=\exp(-10((x-0.5)^2+(y-0.5)^2))\).
\end{example}

\begin{figure}[h!]\centering
\subfigure[True solution $u_0$]{\includegraphics[width=0.32\textwidth]{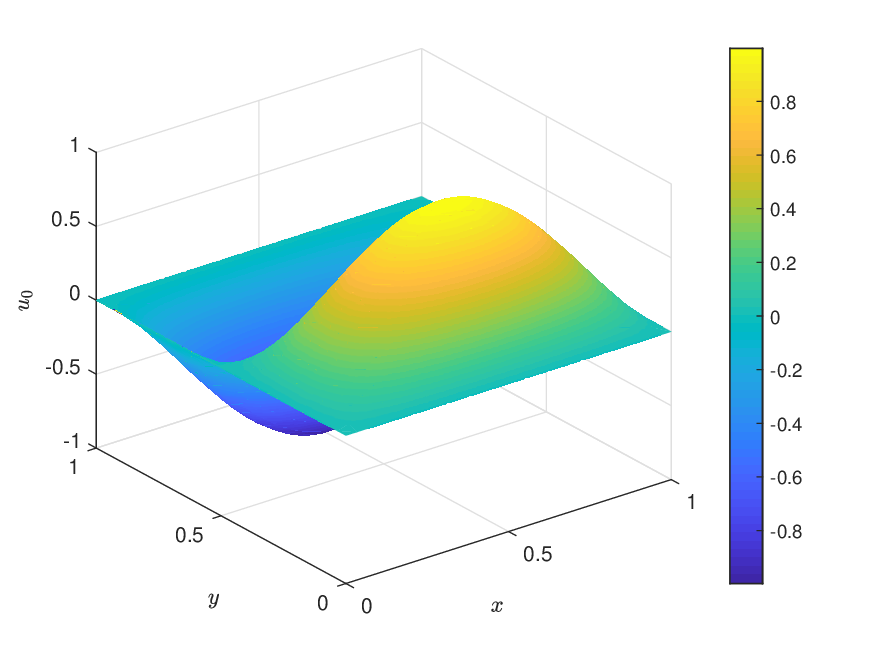}}
\subfigure[Reconstructed solution $u_0^{\rm rec}$]{\includegraphics[width=0.32\textwidth]{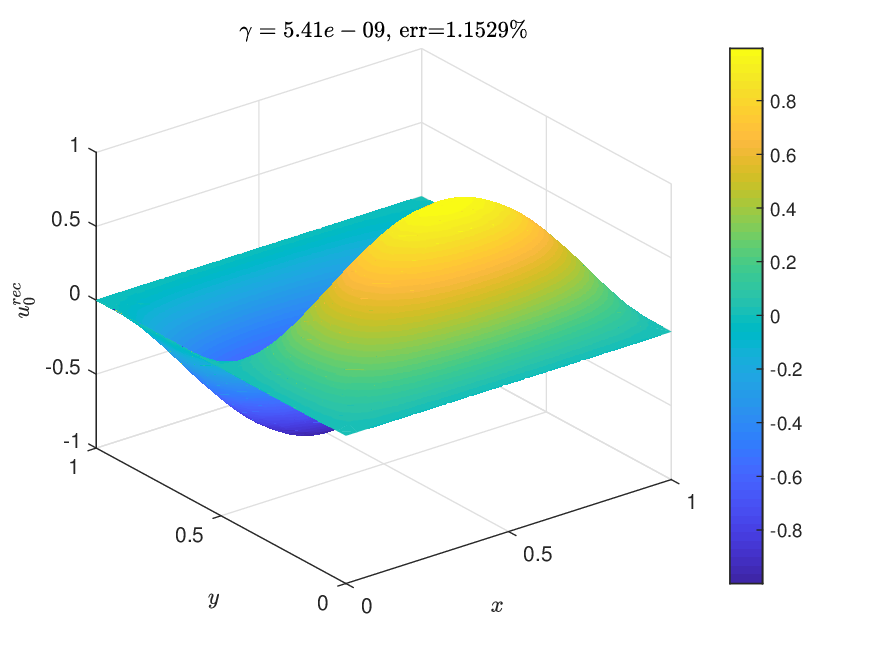}}
\subfigure[Reconstructed solution $u_0^{\rm rec}$]{\includegraphics[width=0.32\textwidth]{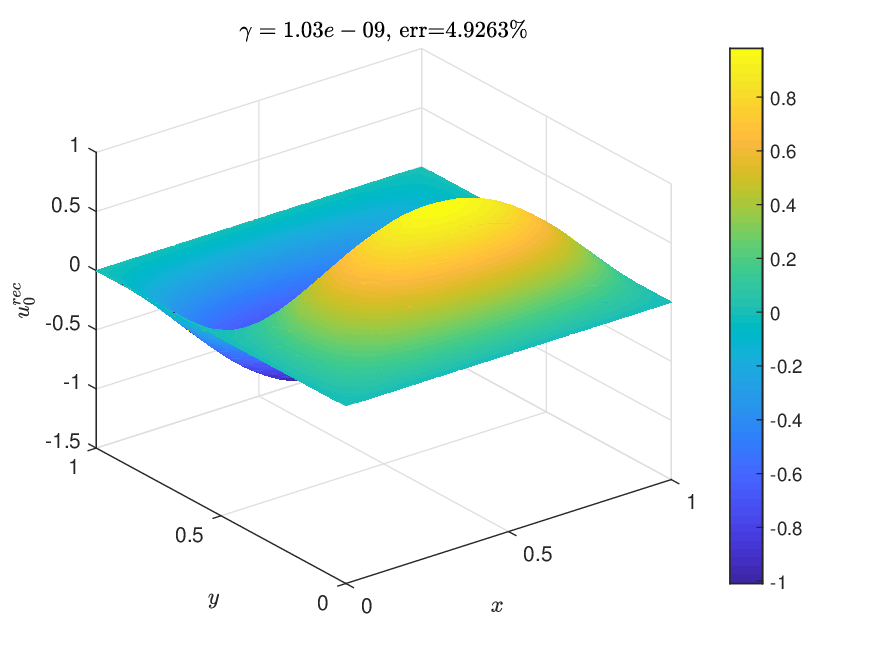}}
\caption{Comparison of the true solution and the reconstructed solution of Example \ref{ex3} for $\omega=\omega_{1}\times\omega_{1}$(the second column) and $\omega=\{x_0\}\times\omega_{1}$(the third column).}\label{fig3}
\end{figure}

\begin{figure}[h!]\centering
\subfigure[True solution $u_0$]{\includegraphics[width=0.32\textwidth]{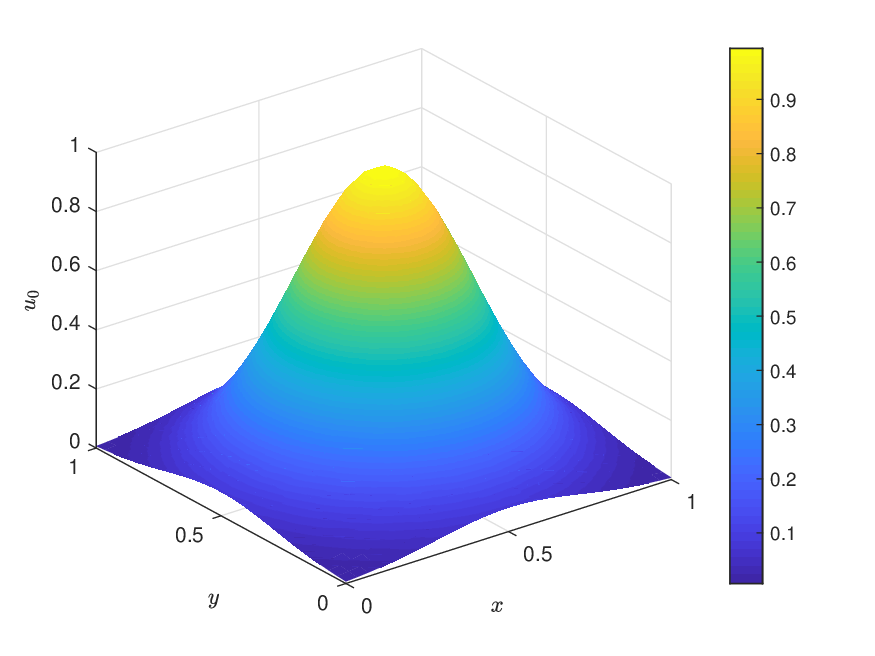}}
\subfigure[Reconstructed solution $u_0^{\rm rec}$]{\includegraphics[width=0.32\textwidth]{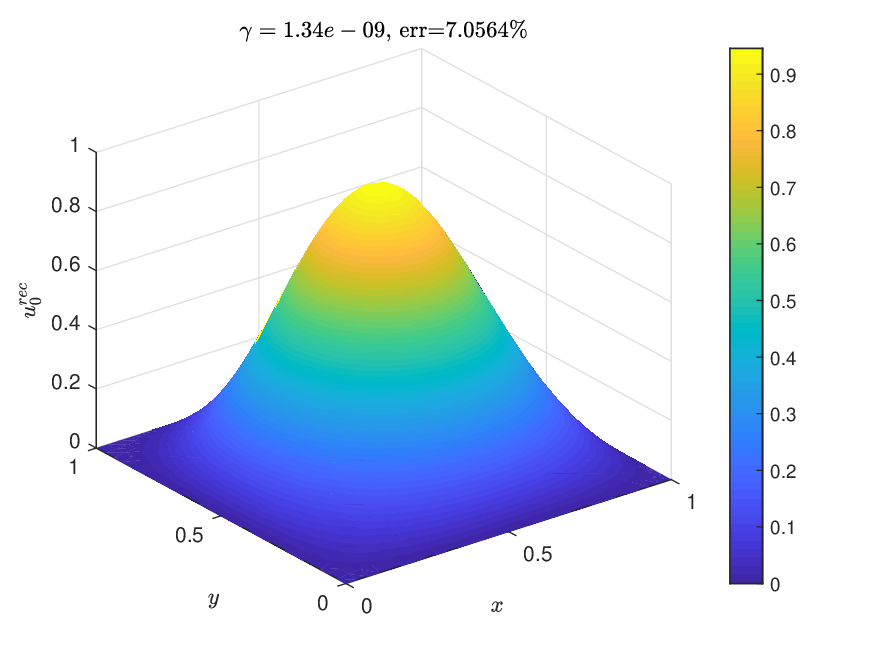}}
\subfigure[Reconstructed solution $u_0^{\rm rec}$]{\includegraphics[width=0.32\textwidth]{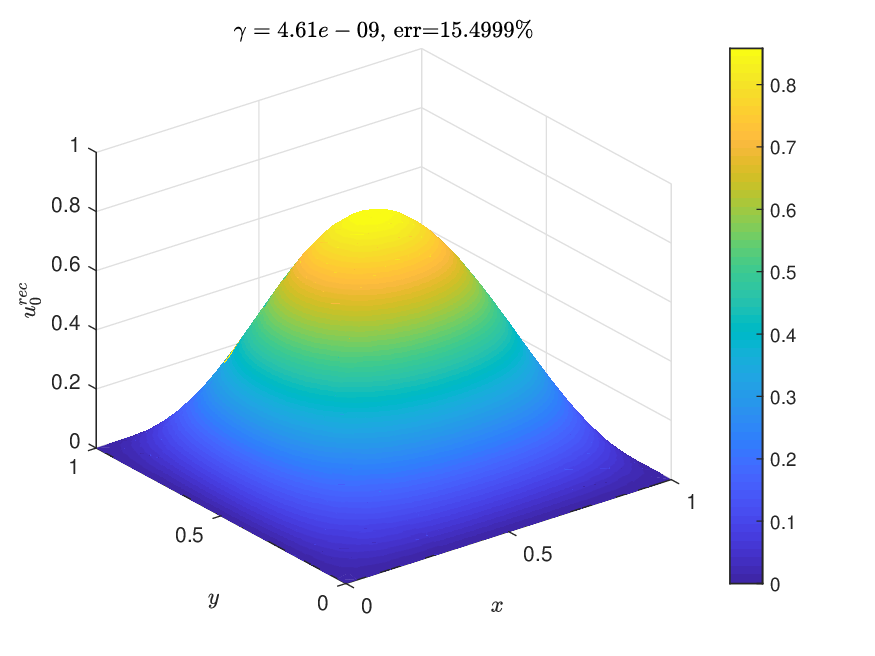}}
\caption{Comparison of the true solution and the reconstructed solution of Example \ref{ex4} for $\omega=\omega_{1}\times\omega_{1}$(the second column) and $\omega=\{x_0\}\times\omega_{1}$(the third column).}\label{fig4}
\end{figure}

The reconstruction results for the two-dimensional case are illustrated in Figures~\ref{fig3}--\ref{fig4}. Configuration~C, consisting of the Cartesian product \(\omega_1\times\omega_1\) with \(36\) observation points, yields excellent accuracy for all two initial conditions, with relative \(L^2\) errors of \(1.1529\%\) (Sine mode) and \(7.0564\%\) (Gaussian). Configuration~D, the mixed product \(\{x_0\}\times\omega_1\) with only \(6\) points, still produces acceptable reconstructions, with errors of  \(4.9263\%\) and \(15.4999\%\) respectively. The slightly larger errors for the Gaussian profile are expected because its broader spectral content makes it more sensitive to noise and truncation. The effect of the source truncation is negligible owing to the high smoothness of the chosen source, confirming that the algorithm remains robust. These results demonstrate that both admissible constructions of the special zero-dimensional observation sets are viable for stable recovery in two dimensions.

Similar to the one‑dimensional case, the initial data in two dimensions can be accurately recovered using only a very small number of observation points. The reduction of the required data is even more pronounced here: while an interior open set in two dimensions would typically require measurements over a positive‑area region, which entails a large number of sensors if one discretises uniformly, Configurations~C and~D use only \(36\) and \(6\) pointwise sensors, respectively. Moreover, as the spatial dimension \(n\) grows, this structural advantage becomes increasingly significant. To resolve spatial scales up to a given mesh size \(h>0\), a uniform sensor grid covering a small interior open set in \(n\) dimensions requires on the order of \(h^{-n}\) sensors. In contrast, our approach uses only \(M_x^n\) points in the Cartesian product case or as few as \(M_x^{\,n-1}\) points in the mixed product case, where \(M_x\) is the number of special points per direction. To make this comparison concrete, suppose one wishes to achieve a spatial resolution of \(h=10^{-2}\) in dimension \(n=3\). A uniform sensor grid covering an interior open set would then require roughly \(h^{-3} = 10^6\) sensors. By contrast, taking \(M_x=6\) special points per direction as in our one‑dimensional experiments, the Cartesian product construction uses \(M_x^3 = 216\) points, while the mixed construction uses only \(M_x^{2}=36\) points. This represents a reduction by four to five orders of magnitude compared with the classical open‑set approach. Consequently, the thin‑set observation strategy offers a drastic reduction in data acquisition costs in higher‑dimensional applications while still providing reliable initial data recovery.

\section{Concluding remarks}\label{sec6}

In this paper, we extended the thin‑set observability theory from the classical heat equation to the fractional heat equation
\[
\partial_t u+(-\Delta)^s u=0,\qquad s\in(1/2,1].
\]
Based on this observability inequality, we studied the inverse problem for the nonhomogeneous equation \(\partial_t u+(-\Delta)^s u=f(x,t)\) with a known source term \(f\), where observations are taken on thin sets. By subtracting the computable Duhamel contribution of \(f\) from the measurements, the recovery of the initial data reduces to the homogeneous case. Under an a priori bound \(\|A^r u_0\|_{L^2(\Omega)}\le M\), we proved uniqueness and logarithmic stability for the recovered initial datum from thin‑set observations, with the logarithmic rate \([\log(e+M/\varepsilon)]^{-r/s}\). This rate reflects the interplay between the backward amplification \(e^{T\Lambda^s}\) of low frequencies and the high‑frequency tail controlled by \(M\Lambda^{-r}\). The geometric exponent \(\beta\) affects only the multiplicative constant through the observability cost, not the logarithmic power.

Numerical experiments in one and two space dimensions confirm that stable recovery from very few, carefully chosen observation points is feasible. A regularised least‑squares algorithm with a spectral penalty on the higher eigenmodes was employed, and a conditional convergence analysis was provided, showing that the reconstruction error is bounded by the sum of a spectral truncation term and a logarithmic term dictated by the continuous stability estimate, and that the error decays logarithmically as the noise level tends to zero. The simulations demonstrate that, as long as the observation set satisfies a suitable spectral inequality, low‑dimensional or structured discrete observations can provide enough information for accurate reconstruction.

\bigskip

\noindent{\bf Acknowledgments.} This work was partially supported by the National Natural Science Foundation of China
(No. 12271277) and the Open Research Fund of Key Laboratory of Nonlinear Analysis \&
Applications (Central China Normal University), Ministry of Education, China. The second author also thanks Ningbo Youth Leading Talent Project (No.2024QL045).


\begin{thebibliography}{99}


\bibitem{AliAziz18}
M. Ali, S. Aziz, S. A. Malik. Inverse source problem for a space-time fractional diffusion equation, {\it Fract. Calc. Appl. Anal.} {\bf 21}(3) (2018): 844-863.

\bibitem{Biccari18}
U. Biccari, {Boundary controllability for a one-dimensional heat equation with a singular inverse-square potential}, {\it Math. Control Relat. Fields} {\bf 9} (1) (2018): 191--219.

\bibitem{BiccariWarmaZuazua18}
U. Biccari, M. Warma, E. Zuazua, {Local regularity for fractional heat equations}, in {\it Recent Advances in PDEs: Analysis, Numerics and Control: In Honor of Prof. Fern\'andez-Cara's 60th Birthday}, Springer, 2018, pp. 233--249.

\bibitem{Bourgeois17}
L. Bourgeois, { Quantification of the unique continuation property for the heat equation}, {\it Math. Control Relat. Fields} {\bf 7} (3) (2017): 347--367.

\bibitem{BucurValdinoci16}
C. Bucur, E. Valdinoci, {\it Nonlocal Diffusion and Applications}, Springer, 2016.

\bibitem{CaffarelliSilvestre07}
L. Caffarelli, L. Silvestre, {An extension problem related to the fractional Laplacian}, {\it Comm. Partial Differential Equations} {\bf 32} (8) (2007): 1245--1260.

\bibitem{CapellaDavilaDupaigne11}
A. Capella, J. D\'avila, L. Dupaigne, Y. Sire, {Regularity of radial extremal solutions for some non-local semilinear equations}, {\it Comm. Pure Appl. Math.} {\bf 36} (8) (2011): 1353--1384..

\bibitem{Coron07}
J.-M. Coron, {\it Control and Nonlinearity}, Amer. Math. Soc., 2007.

\bibitem{DagerZuazua06}
R. Dager, E. Zuazua, {\it Wave Propagation, Observation and Control in $1\text{-}d$ Flexible Multi-Structures}, Springer, 2006.

\bibitem{GarciaTakahashi11}
G. Garc\'ia, T. Takahashi, {\it Inverse problem and null-controllability for parabolic systems}, {\it J. Inverse Ill-Posed Probl.} {\bf 19} (2011): 379--405.

\bibitem{Green25}
A. W. Green, K. Le Balc'h, J. Martin, M. A. Orsoni, {On the dimension of observable sets for the heat equation}, {\it Pure Appl. Anal.} {\bf 7} (3) (2025): 639--660.

\bibitem{HuangWang26}
S. Huang, G. Wang, M. Wang, {Observability Inequality, Log-Type Hausdorff Content and Heat Equations}, {\it Comm. Math. Phys.} {\bf 407} (2) (2026): 33.

\bibitem{Isakov06}
V. Isakov, {\it Inverse Problems for Partial Differential Equations}, 2nd ed., Springer, 2006.

\bibitem{JinRundell15}
B. Jin, W. Rundell, {A tutorial on inverse problems for anomalous diffusion processes}, {\it Inverse Problems} {\bf 31} (3) (2015): 035003.

\bibitem{John60}
F. John, {Continuous dependence on data for solutions of partial differential equations with a prescribed bound}, {\it Comm. Pure Appl. Math.} {\bf 13} (4) (1960): 551--585.

\bibitem{Klibanov06}
M. V. Klibanov, { Estimates of initial conditions of parabolic equations and inequalities via lateral Cauchy data}, {\it Inverse Problems} {\bf 22} (2) (2006): 495--514.

\bibitem{KlibanovTikhonravov07}
M. V. Klibanov, A. V. Tikhonravov, {Estimates of initial conditions of parabolic equations and inequalities in infinite domains via lateral Cauchy data}, {\it J. Differential Equations} {\bf 237} (1) (2007): 198--224.

\bibitem{Koenig20}
A. Koenig, Lack of null-controllability for the fractional heat equation and related equations, {\it SIAM J. Control Optim.} {\bf 58} (6) (2020): 3130--3160.

\bibitem{LebeauRobbiano95}
G. Lebeau, L. Robbiano, {Contr\^ole exact de l'\'equation de la chaleur}, {\it Comm. Partial Differential Equations} {\bf 20} (1--2) (1995): 335--356.


\bibitem{LebeauZuazua98}
G. Lebeau, E. Zuazua, {Null-controllability of a system of linear thermoelasticity}, {\it Arch. Ration. Mech. Anal.} {\bf 141} (4) (1998): 297--329.

\bibitem{LiYamamotoZou08}
J. Li, M. Yamamoto, J. Zou, {Conditional stability and numerical reconstruction of initial temperature}, {\it Commun. Pure Appl. Anal.} {\bf 8} (1) (2008): 361--382.

\bibitem{LiWei18}
Y. S. Li, T. Wei, {An inverse time-dependent source problem for a time–space fractional diffusion equation}, {\it Appl. Math. Comput.} {\bf 336} (2018): 257--271.

\bibitem{LiLiuYamamoto17}
Z. Li, J. Liu, M. Yamamoto, {Inverse problems of determining parameters of the fractional partial differential equations}, in {\it Handbook of Fractional Calculus with Applications}, {\bf 2} 2019: 411--430.

\bibitem{LinRailo25}
Y.-H. Lin, J. Railo, P. Zimmermann, The Calder\'on problem for a nonlocal diffusion equation with time-dependent coefficients, {\it Rev. Mat. Iberoam.} {\bf 41} (3) (2025): 1129--1172.

\bibitem{LuZuazua16}
Q. L\"u, E. Zuazua, {On the lack of controllability of fractional in time ODE and PDE}, {\it Math. Control Signals Systems} {\bf 28} (2) (2016): 10.

\bibitem{Mainardi22}
F. Mainardi, {\it Fractional calculus and waves in linear viscoelasticity: an introduction to mathematical models}, World Scientific, 2022.

\bibitem{MetzlerKlafter00}
R. Metzler, J. Klafter, {The random walk's guide to anomalous diffusion: a fractional dynamics approach}, {\it Phys. Rep.} {\bf 339} (1) (2000): 1--77.

\bibitem{MicuZuazua06}
S. Micu, E. Zuazua, {On the controllability of a fractional order parabolic equation}, {\it SIAM J. Control Optim.} {\bf 44} (6) (2006): 1950--1972.

\bibitem{Miranker61}
W. L. Miranker, {A well-posed problem for the backward heat equation}, {\it Proc. Amer. Math. Soc.} {\bf 12} (2) (1961): 243--247.

\bibitem{PhungWang13}
K. D. Phung, G. Wang, {An observability estimate for parabolic equations from a measurable set in time and its applications}, {\it J. Eur. Math. Soc.} {\bf 15} (2) (2013): 681--703.

\bibitem{RosOtonSerra14}
X. Ros-Oton, J. Serra, {The Dirichlet problem for the fractional Laplacian: regularity up to the boundary}, {\it J. Math. Pures Appl.} {\bf 101} (3) (2014): 275--302.

\bibitem{Russell78}
D. L. Russell, {Controllability and stabilizability theory for linear partial differential equations: recent progress and open problems}, {\it SIAM Rev.} {\bf 20} (4) (1978): 639--739.

\bibitem{SakamotoYamamoto11}
K. Sakamoto, M. Yamamoto, Initial value/boundary value problems for fractional diffusion-wave equations and applications to some inverse problems, {\it J. Math. Anal. Appl.}, {\bf 382 }(1) (2011): 426--447.

\bibitem{TatarUlusoy15}
S. Tatar, S. Ulusoy, An inverse source problem for a one-dimensional space-time fractional diffusion equation, {\it Appl. Anal.}, {\bf 94} (11) (2015): 2233--2244.

\bibitem{TucsnakWeiss09}
M. Tucsnak, G. Weiss, {\it Observation and Control for Operator Semigroups}, Birkh\"auser, 2009.

\bibitem{ZhangJia18}
Y.-X. Zhang, J.X. Jia, L. Yan, Bayesian approach to a nonlinear inverse problem for a time-space fractional diffusion equation. {\it Inverse Problems} {\bf 34} (12) (2018): 125002.

\bibitem{ZhaoLiu14}
J. J. Zhao, S. S. Liu, T. Liu, An inverse problem for space‐fractional backward diffusion problem, {\it Math. Methods Appl. Sci.} {\bf 37}(8) (2014): 1147-1158.

\bibitem{ZhengZhang18}
G. H. Zheng, Q. G. Zhang. Solving the backward problem for space-fractional diffusion equation by a fractional Tikhonov regularization method, {\it Math. Comput. Simul.} {\bf 148} (2018): 37-47.

\bibitem{Zuazua07}
E. Zuazua, {Controllability and observability of partial differential equations: some results and open problems}, in {\it Handbook of differential equations: evolutionary equations}, Vol. 3, North-Holland, 2007, pp. 527--621.

\end{thebibliography}
\end{document}